\documentclass[10pt]{amsart}
\UseRawInputEncoding
\usepackage{amsmath}
\usepackage{amsthm}
\usepackage{amscd}
\usepackage{amsfonts}
\usepackage{amssymb}
\usepackage{graphicx}
\usepackage{color}
\usepackage[bookmarksnumbered, colorlinks, plainpages]{hyperref}
\usepackage{hyperref}
\hypersetup{colorlinks=true,linkcolor=red, anchorcolor=green, citecolor=cyan, urlcolor=red, filecolor=magenta, pdftoolbar=true}

\newtheorem{theorem}{Theorem}[section]
\newtheorem{lemma}[theorem]{Lemma}
\newtheorem{proposition}[theorem]{Proposition}
\newtheorem{corollary}[theorem]{Corollary}
\theoremstyle{definition}
\newtheorem{definition}[theorem]{Definition}
\newtheorem{example}[theorem]{Example}

\newtheorem{remark}[theorem]{Remark}
\numberwithin{equation}{section}
\newcommand{\dd}{\displaystyle}
\begin{document}
\setcounter{page}{1}

\title[Some results]{Some results on higher order isosymmetries in Semi-Hilbertian Spaces}

\author[Rchid RABAOUI]{Rchid RABAOUI}

\address{Department of Mathematics, Faculty of Sciences. Gabes 6072. Tunisia \\
University of Tunis El Manar, Faculty of Sciences of Tunis\\
Mathematical Analysis and Applications Laboratory LR11ES11\\
2092, El Manar I, Tunis, Tunisia.}
\email{rchid.rabaoui@fsg.rnu.tn}


\begin{abstract}
In this paper, we introduce the class of $(A,(m,n))$-isosymmetric operators and we study some of their properties, for a positive semi-definite operator $A$ and $ m,n\in\mathbb{ N}$, which extend, by changing the initial inner product with the semi-inner product induced by $A$, the well-known class of $(m,n)$-isosymmetric operators introduced by Mark Stankus (\cite{mark1}, \cite{mark}). In particular, we characterize a family of $A$-isosymmetric $(2\times2)$ upper triangular operator matrices. Moreover, we show that that if $T$ is $(A,(m,n))$-isosymmetric and if $Q$ is a nilpotent operator of order $r$ doubly commuting with $T$, then $T^p$ is $(A,(m,n))$-isosymmetric symmetric for any $p\in \mathbb{N}$ and $(T +Q)$ is $\big(A,(m+2r -2, n+2r -1)\big)$-isosymmetric. Some properties of the spectrum are also investigated.
\end{abstract} \maketitle

\section{Introduction and preliminaries}
Let $\mathcal{B}(H)$ be the algebra of all bounded linear operators on a separable complex
Hilbert space $H$ with the inner product $\langle \cdot\;|\;\cdot\rangle$. For $T\in \mathcal{B}(H)$, we write $\mathcal{R}(T)$, $\mathcal{N}(T)$, $\sigma(T)$, $\sigma_p(T)$ and $\sigma_{ap}(T)$ for the range, the null space, the spectrum, the point spectrum and the approximate point spectrum of $T$, respectively. A nonzero positif operator $A$ defines a semi-norm $\|u\|_{A}:=\langle Au\;|\;u\rangle^{\frac{1}{2}}.$  Observe that $\parallel\cdot\parallel_{A}$ is a norm if and only if $A$ is injective. For $A\in \mathcal{B}(H)^+$ and $T \in \mathcal{B}(H)$, we consider the following symbols (for more details see again \cite{a.2} and \cite{Le})
\begin{eqnarray*}
\big(\big\{(yx-1)^m\big\}_\mathbf{a}\big)(T,A)
&:=&\bigg(\sum_{k=0}^{m}(-1)^{m-k} \;{m\choose k}\;y^k \mathbf{a} x^{k}\bigg)(T,A),\\
\big(\big\{(x-y)^n\big\}_\mathbf{a}\big)(T,A)
&:=&\bigg(\sum_{k=0}^{n}(-1)^{n-k} \;{n\choose k}\;y^k \mathbf{a} x^{n-k}\bigg)(T,A),\\
\big(\big\{(x+y)^n\big\}_\mathbf{a}\big)(T,A)&=&i^n \big(\big\{(x-y)^n\big\}_\mathbf{a}\big)(iT,A).
\end{eqnarray*}

In $1995$, J. Agler and M. Stankus (\cite{rr01}, \cite{x2}, \cite{x3}) introduced the class of $m$-isometric operators and showed many
properties of such a family. In fact, $T \in \mathcal{B}(H)$ is said to be $m$-isometric if
$$\mathcal{I}^{m}_{I}(T):=\big(\big\{(yx-1)^m\big\}_\mathbf{a}\big)(T,I)=0,\quad \mathcal{I}^{0}_{I}(T)=I.$$
Sid Ahmed et al. $($\cite{X27}$)$ generalized the concept of those operators on a Hilbert space. They introduced the $(A,m)$-isometric operators in a semi-hilbertian space. For $m\geq1$, an operator $T\in \mathcal{B}(H)$ is an $(A,m)$-isometry if
\begin{equation}\label{i}
\mathcal{I}^{m}_{A}(T):=\big(\big\{(yx-1)^m\big\}_\mathbf{a}\big)(T,A)=0,\quad \mathcal{I}^{0}_{A}(T)=A.
\end{equation}
They extend well known properties and showed important results related to such a family
$($\cite{a.2}, \cite{sjn}$)$. The dynamic of the orbits of $(A,m)$-isometries was studied in \cite{faghih} and \cite{X28}. We have shown (see \cite{a.1}) that most of those properties hold true for $(A,n)$-symmetries (resp. skew $(A,n)$-symmetries). An operator $T$ is $(A,n)$-symmetric (resp. skew $(A,n)$-symmetric) if
\begin{eqnarray}\label{s}
\mathcal{S}^{n}_{A}(T)&:=&\big(\big\{(y-x)^n\big\}_\mathbf{a}\big)(T,A)=0,\quad \mathcal{S}^{(0)}_{A}(T)=A\nonumber\\
\big(resp.\;\;\zeta^{n}_{A}(T)&:=&\big(\big\{(y+x)^n\big\}_\mathbf{a}\big)(T,A)=0,\quad \zeta^{0}_{A}(T)=A\big).
\end{eqnarray}
Note that $(A,n)$-symmetries seems to be a natural generalization of $n$-symmetries, which satisfy an identity of the form $\mathcal{S}^{n}_{I}(T)=0,$ equivalently, $T\in Helton_{m}(T^*)$. The study of such a family was initiated by Helton in $1970$ (see \cite{x4.2}).\\

In \cite{mark1}, M. Stankus introduced and intensively studied isosymmetries, namely bounded linear operators satisfying
$$(yx-1)(y-x)(T)=T^{*2}T-T^{*}T^{2}-T^{*}+T=0.$$
In quest of generality, the author in \cite{mark} extends this definition to higher order, and introduces a class of operators termed $(m, n)$-isosymmetries. Such a family of bounded operators satisfies an identity of the form
$$(yx-1)^m(y-x)^n(T)=0.$$
Note that this class includes $n$-symmetries $($i.e. $T$ such that $(y-x)^n(T) = 0)$, $m$-isometries $($i.e. $T$ such that $(yx-1)^m(T) = 0)$ and some classes of non-normal operators (see \cite{mark1}). \\

 Recently, M. Ch\={o} et al. (\cite{t1}, \cite{t2}) introduced $(m,n,C)$-isosymmetries: an operator $T \in \mathcal{B}(H)$ is said to be $(m,n,C)$-isosymmetric if there exists a conjugation $C$ on $H$ such that $\gamma_{m,n}(T,C)=0$, where
$$
\gamma_{m,n}(T,C):=\left\{
                    \begin{array}{ll}
                      \dd\sum_{j=0}^{m}(-1)^{j} \;{m\choose j}\;T^{*m-j}\alpha_n(T,C)CT^{m-j}C, & \hbox{ } \\
                      \dd\sum_{k=0}^{n}(-1)^{k} \;{n\choose k}\;T^{*n-k}\beta_m(T,C)CT^{k}C. & \hbox{ }
                    \end{array}
                  \right.
$$
Recall that a conjugation on $H$ is an antilinear operator $C : H \longrightarrow H$ which satisfies $\langle Cx\mid Cy\rangle = \langle y\mid x\rangle$ for all $x, y \in H$ and $C^2 = I$.\\

 In the light of $(m,n)$-isosymmetric and $(m,n,C)$-isosymmetric operators, using the identities (\ref{i}) and (\ref{s}), we aim in this paper to introduce the higher order isosymmetries in semi-Hilbertian spaces and to initiate this study by giving some of their basic properties and results. For $A\in \mathcal{B}(H)^{+},\;T \in \mathcal{B}(H)$ and $ m,\,n\in \mathbb{N}$, set
$$
\Omega_{A}^{m,n}(T):=\left\{
                         \begin{array}{ll}
                          \dd\sum_{j=0}^{m}(-1)^{m-j} \;{m\choose j}\;T^{*j}\mathcal{S}^{n}_{A}(T)T^{j}, & \hbox{} \\
                          \dd\sum_{k=0}^{n}(-1)^{n-k} \;{n\choose k}\;T^{*k}\mathcal{I}^{m}_{A}(T)T^{n-k}, & \hbox{}
                         \end{array}
                       \right.,
\Lambda_{A}^{m,n}(T):=\left\{
                          \begin{array}{ll}
                           \dd\sum_{k=0}^{m}(-1)^{m-k} \;{m\choose k}\;T^{*k}\zeta^{n}_{A}(T)T^{k}, & \hbox{} \\
                           \dd\sum_{k=0}^{n}{n\choose k}\;T^{*k}\mathcal{I}^{m}_{A}(T)T^{n-k}. & \hbox{}
                          \end{array}
                        \right.$$
We define $(A;(m, n))$-isosymmetric and skew $(A;(m, n))$-isosymmetric operators as follows:
\begin{enumerate}
\item $T$ is said to be $(A;(m, n))$-isosymmetric if $\Omega_{A}^{m,n}(T)=0$.  Let $(A; (m, n))$ denote the set of $(A; (m, n))$-isosymmetries. Then
$(A; (m, 0))$ is the set of $(A, m)$-isometries and $(A; (0, n))$ is the set of $(A, n)$-symmetries.
\item $T$ is said to be skew $(A;(m, n))$-isosymmetric if $\Lambda_{A}^{m,n}(T)=0$.
\end{enumerate}

\begin{remark}
\noindent
\begin{enumerate}
\item Recall by definition
$$T\in (A; (m, n))\Longleftrightarrow \mathcal{S}^{n}_{A}(T)\in (I; (m, 0))   \Longleftrightarrow \mathcal{I}^{m}_{A}(T)\in (I; (0, n)) .$$
Note that:
\begin{enumerate}
\item If $\mathcal{S}^{n}_{A}(T)\geq0$, then $T\in (A; (m, n))$
if and only if $T\in (\mathcal{S}^{n}_{A}(T); (m, 0))$.
\item If $\mathcal{I}^{m}_{A}(T)\geq0$, then $T\in (A; (m, n))$
if and only if $T\in (\mathcal{I}^{m}_{A}(T); (0, n))$.
\end{enumerate}
\item An $(A;(1, 1))$-isosymmetry is an $A$-isosymmetry, and if $A=I$ then an $(A;(m, n))$-isosymmetry is an $(m, n)$-isosymmetry.
\item If $T\in (A; (m, 0))$, then $T\in(A; (m, n^\prime))$ for all integer $n^\prime\geq0$.
\item If $T\in (A; (0, n))$, then $T\in(A; (m^\prime, n))$ for all integer $m^\prime\geq0$.
\item If $S\in \mathcal{B}(H)$ is unitarly equivalent to $T$, then $S\in(A; (m, n))$ if and only if $T\in(VAV^*; (m, n))$, with $V$ is unitary. Indeed, since $T$ and $S$ are unitary equivalent, there exists $V\in B(H)$ unitary with $\mathcal{S}^{n}_{A}(S)=V^* \mathcal{S}^{n}_{VAV^*}(T)V.$
It holds
\begin{eqnarray*}
\Omega_{A}^{m,n}(S)
&=&V^* \big(\Omega_{VAV^*}^{m,n}(T)\big)V.
\end{eqnarray*}
\item From the following identity
$$ \big(\big\{(x+y)^n\big\}_\mathbf{a}\big)(T,A)=i^n \big(\big\{(x-y)^n\big\}_\mathbf{a}\big)(iT,A)$$
it can be seen that $T$ is skew $(A; (m, n))$-isosymmetric if and only if $iT \in (A; (m,n))$. So, several parallel results on skew $(A; (m, n))$-isosymmetries  follow from the results of $(A; (m, n))$-isosymmetries.
\item The class of $(A;(m, n))$-isosymmetric (resp. skew $(A;(m, n))$-isosymmetric) operators is closed in norm, that is: let $T \in \mathcal{B}(H)$ and $A\in \mathcal{B}(H)^+$, if $(T_p)_p$ is a sequence of $(A;(m, n))$-isosymmetric $($resp. skew $(A;(m, n))$-isosymmetric$)$ operators such that $\dd\lim_{p\rightarrow+\infty}\|T_p - T\| = 0$, then $T$ is also $(A;(m, n))$-isosymmetric $($resp. skew $(A;(m, n))$-isosymmetric$)$. Indeed, assume
$T_p \longrightarrow T$ in norm as $p\longrightarrow\infty$, then $0 =\Omega_{A}^{m,n}(T_p)\longrightarrow \Omega_{A}^{m,n}(T)$. Thus $\Omega_{A}^{m,n}(T)=0$ and $T \in (A; (m, n))$.
\end{enumerate}
\end{remark}

Let $$\mathcal{B}_{A}(H):=\big\{T\in \mathcal{B}(H):\mathcal{R}(T^{*}A)\subset \mathcal{R}(A)\big\}.$$
Recall that, if $T \in  \mathcal{B}_{A}(H) $, then $T$ admits an $A$-adjoint operator $T^\sharp$. This operator verifies
$$AT^\sharp =T^* A,\;\mathcal{R}(T^\sharp)\subseteq \overline{\mathcal{R}(A)}\;\hbox{and}\;\mathcal{N}(T^\sharp )=\mathcal{N}(T^* A).$$
\begin{enumerate}
\item Let $T \in  \mathcal{B}_{A}(H) $. Then, $T$ is $(A;(m, n))$-isosymmetric if and only if the following operator identity holds
$$\sum_{j=0}^{m}\sum_{k=0}^n(-1)^{m+n-(j+k)} \;{m\choose j}{n \choose k}\;T^{\sharp k+j} T^{n+j-k} =0.$$
%
\item Under some conditions, the classes of $(m,n)$-isosymmetries and $(A,(m,n))$-isosymmetries coincide. Let $T \in \mathcal{B}(H)$ and $A$ be injective. Then, the following claims hold.
\begin{enumerate}
\item If $A$ and $T$ commute, then $T$ is $(A,(m,n))$-isosymmetric if and only if $T$ is $(m,n)$-isosymmetric.
\item If $T$ is $A$-normal $($i.e $T^{\sharp}T=TT^{\sharp})$, then $T$ is $(A,(m,n))$-isosymmetric if and only if $T^\sharp$ is $(A,(m,n))$-isosymmetric.
\end{enumerate}
\end{enumerate}


\begin{example}
\noindent
\begin{enumerate}
\item The identity operator and the orthogonal projection on $\overline{R(A)}$ are in $(A;(m, n))$, for every $m,n\in \mathbb{N}$.
\item Let $A\in \mathcal{B}(H)^+$ and $T\in \mathcal{B}(H)$. If $T$ is nilpotent
of order $r > 2$, then $T\in(A; (2r-2, 2r-1)$. Indeed, since $T$ is nilpotent of order $r$, it gives that $T^*j=T^j=0$ for all $j\geq r$. Then, since $\max\{k+j,(2r-1)-j+k\}\geq r$ for any $j\in \{0,1,\cdots,2r-1\}$ and $k\in \{0,1,\cdots,2r-2\},$ we get
$$\Omega_{A}^{2r-2,2r-1}(T)=\sum_{k=0}^{2r-2}\sum_{j=0}^{2r-1}(-1)^{4r-3-k-j} \;{2r-2\choose k}{2r-1\choose j}\;T^{*k+j}AT^{2r-1-j+k}=0.$$
\item Generally, an $(A;(m, n))$-isosymmetry is not necessarily an $(m, n)$-isosymmetry. Let $H=\mathbb{C}^{2}$ with norm ${\|(x,y)\|}^2=|x|^2 + |y|^2$. Let $A = \left(
           \begin{array}{cc}
             1 & 1 \\
             1 & 1 \\
           \end{array}
         \right)$ and $T = \left(
            \begin{array}{cc}
              0 &  \frac{-1}{2} \\
              0 &  \frac{1}{2} \\
            \end{array}
          \right)
$. 
Since $T\in (A; (0, 1))$, it holds that $T\in (A; (1, 1))$. Moreover, $T$ is not isosymmetric.
\item An operator $T\in \mathcal{B}(H)$ can be in $(A; (m, n))$, without being in $(A; (m, 0))$ or in $(A; (0, n))$. In fact, Let $H=\mathbb{C}^{3}$ with norm ${\|(x,y,z)\|}^2=|x|^2 + |y|^2 + |z|^2$,
$A=\left(
     \begin{array}{ccc}
       0 & 0 & 0 \\
       0 & 1 & 1 \\
       0 & 1 & 1 \\
     \end{array}
   \right)
$ and
$T=\left(
     \begin{array}{ccc}
       0 & 0 &-1 \\
       1 & 0 & 0 \\
       0 & 1 & 0 \\
     \end{array}
   \right)
$. Then, a simple calculation gives $T\in (A; (1, 1))$ but neither in $(A; (1, 0))$, nor in $(A; (0, 1))$.
\end{enumerate}
\end{example}

We now outline the plan of this paper. In Section \ref{S3}, we restrict our selves to investigate general properties of $(A,(m,n))$-isosymmetric operators. In particular, we show that every integer power of an $(A,(m,n))$-isosymmetry is also $(A,(m,n))$-isosymmetric. Some relations  between $(A,(m,n))$-isosymmetries (resp. skew $(A,(m,n))$-isosymmetries) and their exponential operators are established. At the end of the section, we study some spectral properties. In Section $3$, we investigate the stability of an $(A,(m,n))$-isosymmetric (resp. skew $(A,(m,n))$-isosymmetric) operator $T$ under some perturbation by a nilpotent operator. More precisely, we show that if $T$ is $(A,(m,n))$-isosymmetric and if $Q$ is a nilpotent operator of order $r$ doubly commuting with $T$, then $T^p$ is $(A,(m,n))$-isosymmetric symmetric for any $p\in \mathbb{N}$ and $(T +Q)$ is $\big(A,(m+2r -2, n+2r -1)\big)$-isosymmetric. In the closing section, we characterize a family of $A$-isosymmetric upper triangular operator matrices. We first show that, under suitable conditions on $N,E,X\in \mathcal{B}(H)$, if $N$ is $A$-isometric, then $\textbf{T}=\left(
                                                                                            \begin{array}{cc}
                                                                                              N & E \\
                                                                                              0 & X \\
                                                                                            \end{array}
                                                                                          \right)
$ on $\textbf{H}=H\oplus H$ is $A$-isosymmetric if and only if $X$ is $A$-isosymmetric. Moreover, if $N$ is $A$-symmetric and $NE=EX$, then, $\textbf{T}$ is $A$-isosymmetric if and only if $X$ is $A$-isosymmetric and $AE=ANEX$. We then prove that if $\mathcal{M}=\ker\big(\mathcal{I}^{1}_{A}(\textbf{T})\big)$ is invariant for $A$ and for $\textbf{T},$ then $\textbf{T}=\left(
                                                                                            \begin{array}{cc}
                                                                                              N & E \\
                                                                                              0 & X \\
                                                                                            \end{array}
                                                                                          \right)$ on $\mathcal{M}\oplus \mathcal{M}^{\perp}$, with $N$ is an $A_{|\mathcal{M}}$-isometry and $E^*A_{|\mathcal{M}}N=0.$

\section{ Basic Properties of $(A;(m, n))$-isosymmetries}
\label{S3}

The main purpose of this section is to develop some properties for the classes of $(A;(m, n))$-isosymmetries and skew $(A;(m, n))$-isosymmetries on Hilbert spaces.

\begin{proposition}\label{waw}
Let $T \in \mathcal{B}(H)$ and $A\in \mathcal{B}(H)^+$. Then, the following assertions hold.
\begin{enumerate}
\item $T$ is $(A;(m, n))$-isosymmetric if and only if
\begin{eqnarray}\label{007}
\langle\Omega_{A}^{m,n}(T)x\,|\,x\rangle&=&\sum_{k=0}^{m}(-1)^{m-k} \;{m\choose k}\;\langle\mathcal{S}^{n}_{A}(T)T^{k}x\,|\,T^{k}x\rangle\nonumber\\
&=&\sum_{k=0}^{n}(-1)^{n-k} \;{n\choose k}\;\langle\mathcal{I}^{m}_{A}(T)T^{n-k}x\,|\,T^{k}x\rangle=0,\quad \forall\,x\in H.
\end{eqnarray}
\item If $T$ is $(A;(m, n))$-isosymmetric, then $T$ is $(A;(m^\prime , n^\prime ))$-isosymmetric $(m^\prime \geq m,n^\prime \geq n).$
\item If $T$ is $(A,n)$-symmetric $($resp. skew $A$-symmetric$)$, then $(T-s)$ is $(A;(m, n))$-isosymmetric $($resp. skew $(A;(m, n))$-isosymmetric$)$ for $s\in \mathbb{R}$.
\end{enumerate}
\end{proposition}

\begin{proof}
\begin{enumerate}
\item A simple computation shows that
$$\big(\Omega_{A}^{m,n}(T)\big)^{*}=
\left\{
  \begin{array}{ll}
    \Omega_{A}^{m,n}(T), & \hbox{if n is even;} \\
    -\Omega_{A}^{m,n}(T), & \hbox{if n is odd.}
  \end{array}
\right.
$$
If $n$ is odd then $\big(i\Omega_{A}^{m,n}(T)\big)^{*}=i\Omega_{A}^{m,n}(T)$ and, consequently, $\mathcal{K}:=i\Omega_{A}^{m,n}(T)$ is a symmetric operator, which allows to conclude.
\item Since the following identities hold
\begin{eqnarray*}
\mathcal{I}^{m+1}_{A}(T)
 &=&\big( y\big\{(yx-1)^{m}\big\}_\mathbf{a}x-\big\{(yx-1)^{m}\big\}_\mathbf{a}\big)(T,A)
=T^* \mathcal{I}^{m}_{A}(T)-\mathcal{I}^{m}_{A}(T)T,\\
\mathcal{S}^{n+1}_{A}(T) &=& \big(y\big\{(y-x)^{n}\big\}_\mathbf{a}-\big\{(y-x)^{n}\big\}_\mathbf{a}x\big)(T,A)
= T^* \mathcal{S}^{n}_{A}(T)T-\mathcal{S}^{n}_{A}(T),
\end{eqnarray*}
we obtain immediately
\begin{eqnarray*}
\Omega_{A}^{m+1,n}(T)&=&T^* \Omega_{A}^{m,n}(T)-\Omega_{A}^{m,n}(T)T,\\
\Omega_{A}^{m,n+1}(T)&=&T^* \Omega_{A}^{m,n}(T)T-\Omega_{A}^{m,n}(T).
\end{eqnarray*}
The above two equalities completes the proof.
\item The following identity follows from the multinomial theorem
\begin{eqnarray*}
&&\big((y-s)(x-s)-1\big)^m\big\{\big((y-s)-(x-s)\big)^n\big\}_\mathbf{a}\\
&=&\big((yx-1)+(y -s)(-s)+(-s x)\big)^m\big\{(y-x)^n\big\}_\mathbf{a}\\
&=&\bigg(\sum_{\alpha_1+\alpha_2+\alpha_3=m}{m\choose \alpha_1,\alpha_2,\alpha_3}(y -s)^{\alpha_2} (-s)^{\alpha_3+\alpha_2} (yx-1)^{\alpha_1} x^{\alpha_3}\bigg)\big\{(y-x)^n\big\}_\mathbf{a}\\
&=&\sum_{\alpha_1+\alpha_2+\alpha_3=m}{m\choose \alpha_1,\alpha_2,\alpha_3}(y -s)^{\alpha_2} (-s)^{\alpha_3+\alpha_2} (yx-1)^{\alpha_1}\big\{(y-x)^n\big\}_\mathbf{a} x^{\alpha_3}\\
&=&\sum_{k=0}^{m}\sum_{j=0}^{m-k}\;{m\choose k}{m-k\choose j}\;(y -s)^{k} (-s)^{k+j} (yx-1)^{m-k-j}\big\{(y-x)^n\big\}_\mathbf{a} x^{j}.
\end{eqnarray*}
Moreover, we have
\begin{eqnarray*}
&&\big((y-s)(x-s)-1\big)^m\big\{\big((y-s)+(x-s)\big)^n\big\}_\mathbf{a}\\
&=&\sum_{k=0}^{m}\sum_{j=0}^{m-k}\sum_{i=0}^{n-i}\;{m\choose k}{m-k\choose j}{n\choose i}\;(y -s)^{k} (-s)^{k+j} (yx-1)^{m-k-j}\big\{(y+x)^i\big\}_\mathbf{a}(-2s)^{n-i} x^{j}.
\end{eqnarray*}
\noindent By using these identities with $x$ replaced by $T$, $y$ replaced by $T^*$ and $a$ replaced by $A$, we obtain
\begin{equation}
\Omega_{A}^{m,n}(T-s)=\sum_{k=0}^{m}\sum_{j=0}^{m-k}\;{m\choose k}{m-k\choose j}\;(T^* -s)^{k} (-s)^{k+j} \Omega_{A}^{m-k-j,n}(T) T^{j},
\end{equation}
\begin{equation}
\Lambda_{A}^{m,n}(T-s)=\sum_{k=0}^{m}\sum_{j=0}^{m-k}\sum_{i=0}^{n-i}\;{m\choose k}{m-k\choose j}{n\choose i}\;(T^* -s)^{k} (-s)^{k+j} \Lambda_{A}^{m-k-j,t}(T)(-2s)^{n-i} T^{j}.
\end{equation}
This completes the proof.
\end{enumerate}
\end{proof}

\begin{proposition}
Let $T \in \mathcal{B}(H)$ and $A\in \mathcal{B}(H)^+$. Then, the following statements hold.
\begin{enumerate}
\item If $T$ is $(A,m)$-isometric and $\Omega_{A}^{m-1,n}(T)=0$, then $T$ is $\Big(\frac{1}{(m-1)!}\mathcal{I}_{A}^{(m-1)}(T),n\Big)$-symmetric.
\item If $T$ is $(A,n)$-symmetric and $\Omega_{A}^{m,n-1}(T)=0$, then $T$ is $\Big(\frac{(-i)^{n-1}}{(n-1)!}\mathcal{S}_{A}^{(n-1)}(T),m\Big)$-isometric.
\end{enumerate}
\end{proposition}

\begin{proof}
\begin{enumerate}
\item Let consider the following identity
$$\Omega_{A}^{m,n}(T)=\sum_{k=0}^{n}(-1)^{n-k} \;{n\choose k}\;T^{*k}\mathcal{I}^{m}_{A}(T)T^{n-k} .$$
Since $T$ is $(A,m)$-isometric, $\frac{1}{(m-1)!}\mathcal{I}^{m-1}_{A}(T)$ is positive and $T$ is $(A;(m, n))$-isosymmetric. Moreover, we have
$$\frac{1}{(m-1)!}\Omega_{A}^{m-1,n}(T)=\sum_{k=0}^{n}(-1)^{n-k} \;{n\choose k}\;T^{*k}\Big\{\frac{1}{(m-1)!}\mathcal{I}^{m-1}_{A}(T)\Big\}T^{n-k}=0.$$
Then, $\frac{(-i)^{n-1}}{(m-1)!(n-1)!}\Omega_{A}^{m-1,n-1}(T)$ is positive and $T$ is $\Big(\frac{1}{(m-1)!}\mathcal{I}_{A}^{(m-1)}(T),n\Big)$-symmetric.
\item Let consider the following identity
$$\Omega_{A}^{m,n}(T)=\sum_{k=0}^{m}(-1)^{m-k} \;{m\choose k}\;T^{*k}\mathcal{S}^{n}_{A}(T)T^{k}.$$
Since $T$ is $(A,n)$-symmetric, $\frac{(-i)^{n-1}}{(n-1)!}\mathcal{S}_{A}^{(n-1)}(T)$ is a positive operator and $T$ is $(A;(m, n))$-isosymmetric.
 On the other hand, we have
$$\frac{(-i)^{n-1}}{(n-1)!}\Omega_{A}^{m,n-1}(T)=\sum_{k=0}^{m}(-1)^{m-k} \;{m\choose k}\;T^{*k}\Big\{\frac{(-i)^{n-1}}{(n-1)!}\mathcal{S}^{n-1}_{A}(T)\Big\}T^{k}=0 .$$
Hence, $\frac{(-i)^{n-1}}{(m-1)!(n-1)!}\Omega_{A}^{m-1,n-1}(T)$ is positive and $T$ is $\Big(\frac{(-i)^{n-1}}{(n-1)!}\mathcal{S}_{A}^{(n-1)}(T),m\Big)$-isometric.
\end{enumerate}
\end{proof}

\begin{corollary}Let $T \in \mathcal{B}(H)$ and $A\in \mathcal{B}(H)^+$. Assume that $T$ is $(A,n)$-symmetric and $\Omega_{A}^{m,n-1}(T)=0$. Then, the following properties hold.
\begin{enumerate}
\item If $n=2p+1$, then $\Omega_{A}^{m-1,2p}(T)\geq0$ if $p$ is even.
\item If $n$ is even, then $T$ is $(A;(m-1, n-1))$-isosymmetric.
\end{enumerate}
\end{corollary}
\begin{proof}
\begin{enumerate}
\item Assume $n=2p+1$. Since $T$ is $(A,n)$-symmetric and $\Omega_{A}^{m,n-1}(T)=0$, we have
$$0\leq\frac{(-i)^{n-1}}{(m-1)!(n-1)!}\Omega_{A}^{m-1,n-1}(T)=\frac{(-1)^p}{(m-1)!(2p)!}\Omega_{A}^{m,2p}(T)$$
which allows to conclude.
\item Assume that $n$ is even. Since $T$ is $(A,n)$-symmetric, by Theorem $2.5$-\cite{a.1} we have $\mathcal{S}^{n-1}_{A}(T)=0$. Hence, $\Omega_{A}^{m-1,n-1}(T)= 0$, that is $T$ is $(A,(m-1, n-1))$-isosymmetric.
\end{enumerate}
\end{proof}

\begin{theorem}\label{therm11}
Let $T \in \mathcal{B}(H)$ and $A\in \mathcal{B}(H)^+$. Then, the following assertions hold.
\begin{enumerate}
\item If $T$ is invertible, then $T$ is $(A;(m, n))$-isosymmetric $($resp. skew $(A;(m, n))$-isosymmetric$)$ if and only if $T^{-1}$ is
 $(A;(m, n))$-isosymmetric $($resp. skew $(A;(m, n))$-isosymmetric$)$.
\item If $T$ is $(A;(m, n))$-isosymmetric, then $T^k$ is $(A;(m , n ))$-isosymmetric $(\forall\,k\in \mathbb{N})$.
\item If $T$ is skew $(A;(m, n))$-isosymmetric, then $T^k$ is skew $(A;(m , n ))$-isosymmetric $(k$ odd$)$.
\end{enumerate}
\end{theorem}

\begin{proof}
\begin{enumerate}
\item Assume that $T$ is $(A;(m, n))$-isosymmetric and invertible. Since if $m$ is even, then $(-1)^p = (-1)^{m-p}$ and if $m$ is odd, then
$(-1)^p = (-1)^{m+1-p}$, we obtain
\begin{eqnarray*}
0&=&(T^{-1})^{*m+n}\Omega_{A}^{m,n}(T)(T^{-1})^{m+n}\\
&=&\sum_{j=0}^{m}(-1)^{m-j} \;{m\choose j}\;T^{*-m-n}T^{*j}\mathcal{S}^{n}_{A}(T)T^{j}T^{-m-n}\\
&=&\sum_{j=0}^{m}(-1)^{m-j} \;{m\choose m-j}\;T^{*-(m-j)}\big(T^{*-n}\mathcal{S}^{n}_{A}(T)T^{-n}\big)T^{-(m-j)}\\
&=&\sum_{p=0}^{m}(-1)^{p} \;{m\choose p}\;(T^{-1})^{*p}\big((T^{-1})^{*n}\mathcal{S}^{n}_{A}(T)(T^{-1})^{n}\big)(T^{-1})^{-p}\\
&=&\left\{
  \begin{array}{ll}
    \Omega_{T^{-1}}^{m,n}(A), & \hbox{if $m$ is even;} \\
    -\Omega_{T^{-1}}^{m,n}(A), & \hbox{if $m$ is odd.}
  \end{array}
\right.
\end{eqnarray*}
By applying a similar argument, we can show the corresponding property for skew $(A;(m, n))$-isosymmetric operators.
\item Note that, for $k\in \mathbb{N}$ and for some constants $\alpha_l$ and $\beta_j$, we have
\begin{eqnarray*}
&&(y^k x^k-1)^m\big\{(y^k -x^k)^n\big\}_\mathbf{a}\\
&=&\big((yx - 1)(y^{k-1}x^{k-1} + y^{k-2}x^{k-2} + \cdots + 1)\big)^m \big\{\big((y - x)(y^{k-1} + y^{k-2}x + \cdots + x^{k-1})\big)^n\big\}_\mathbf{a}\\
&=&\bigg(\sum_{l=0}^{m(k-1)}\alpha_{l}y^{m(k-1)-l}(yx-1)^m x^{m(k-1)-l}\bigg)\bigg\{\sum_{j=0}^{n(k-1)}\beta_{j}y^{n(k-1)-j}(y-x)^n x^j\bigg\}_\mathbf{a}\\
&=&\bigg(\sum_{l=0}^{m(k-1)}\alpha_{l}y^{m(k-1)-l}(yx-1)^m x^{m(k-1)-l} \bigg)\bigg(\sum_{j=0}^{n(k-1)}\beta_{j}y^{n(k-1)-j}\big\{(y-x)^n\big\}_\mathbf{a} x^j\bigg)\\
&=&\sum_{l=0}^{m(k-1)}\sum_{j=0}^{n(k-1)}\alpha_{l}\beta_{j}y^{m(k-1)-l}y^{n(k-1)-j}(yx-1)^m \big\{(y-x)^n\big\}_\mathbf{a}x^jx^{m(k-1)-l}.
\end{eqnarray*}
The above identity gives
\begin{eqnarray*}
\Omega_{A}^{m,n}(T^k)&=&\big((y^k x^k-1)^m\big\{(y^k -x^k)^n\big\}_\mathbf{a}\big)\big(T^*,T,A\big)\\
&=&\sum_{l=0}^{m(k-1)}\sum_{j=0}^{n(k-1)}\alpha_{l}\beta_{j}T^{*m(k-1)-l+n(k-1)-j}\Omega_{A}^{m,n}(T)T^{j+m(k-1)-l},
\end{eqnarray*}
and the desired claim follows from that.
\item For an odd integer $k$ and for some constants $\alpha_l$ and $\xi_j$, the following formula holds true
\begin{eqnarray*}
&&(y^k x^k-1)^m\big\{(y^k +x^k)^n\big\}_\mathbf{a}\\
&=&\big((yx - 1)(y^{k-1}x^{k-1} + y^{k-2}x^{k-2} + \cdots + 1)\big)^m \bigg\{\Big((y+x)(y^{k-1}-y^{k-2}x+\cdots-yx^{k-2}+x^{k-1}\Big)^n\bigg\}_\mathbf{a}\\
&=&\bigg(\sum_{l=0}^{m(k-1)}\eta_{l}y^{m(k-1)-l}(yx-1)^m x^{m(k-1)-l}\bigg)\bigg\{\sum_{j=0}^{n(k-1)}\xi_{j}y^{n(k-1)-j}(y+x)^n x^j\bigg\}_\mathbf{a}\\
&=&\bigg(\sum_{l=0}^{m(k-1)}\eta_{l}y^{m(k-1)-l}(yx-1)^m x^{m(k-1)-l} \bigg)\bigg(\sum_{j=0}^{n(k-1)}\xi_{j}y^{n(k-1)-j}\big\{(y+x)^n\big\}_\mathbf{a} x^j\bigg)\\
&=&\sum_{l=0}^{m(k-1)}\sum_{j=0}^{n(k-1)}\eta_{l}\xi_{j}y^{m(k-1)-l}y^{n(k-1)-j}(yx-1)^m \big\{(y+x)^n\big\}_\mathbf{a}x^jx^{m(k-1)-l}.
\end{eqnarray*}
It holds from the above identity
\begin{eqnarray*}
\Lambda_{A}^{m,n}(T^k)&=&\big((y^k x^k-1)^m\big\{(y^k +x^k)^n\big\}_\mathbf{a}\big)\big(T^*,T,A\big)\\
&=&\sum_{l=0}^{m(k-1)}\sum_{j=0}^{n(k-1)}\eta_{l}\xi_{j}T^{*m(k-1)-l+n(k-1)-j}\Lambda_{A}^{m,n}(T)T^{j+m(k-1)-l}.
\end{eqnarray*}
If $T$ is skew $(A;(m, n))$-isosymmetric, that is $\Lambda_{A}^{m,n}(T)=0$, then $\Lambda_{A}^{m,n}(T^k)=0$. Hence, $T^k$ is skew $(A;(m , n ))$-isosymmetric.
\end{enumerate}
\end{proof}

In the following corollary, we state some immediate consequences of Theorem \ref{therm11}.
\begin{corollary}Let $T \in \mathcal{B}(H)$ and $A\in \mathcal{B}(H)^+$. Then, the following statements hold.
\begin{enumerate}
\item If $T$ is $(A,m)$-isometric, then $T^k$ is $(A,m)$-isometric for any $k\in \mathbb{N}$.
\item If $T$ is $(A,n)$-symmetric, then $T^k$ is $(A,n)$-symmetric for any $k\in \mathbb{N}$.
\item If $T$ is skew $(A,n)$-symmetric, then $T^k$ is skew $(A,n)$-symmetric for $k$ odd.
\end{enumerate}
\end{corollary}

\begin{proposition}\label{lem2.3}
Let $T$ be $(A,(m, n))$-isosymmetric. Then, the following statements hold.
\begin{enumerate}
\item If $p\geq m,$ then
\begin{equation}\label{(2.17)}
\sum_{k=0}^p(-1)^{p-k}\binom{p}{k}k^i T^{*k}\mathcal{S}^{n}_{A}(T)T^{k}=0,\quad i=0,1,\cdots,p-m.
\end{equation}
\item If $p\geq n,$ then
\begin{equation}\label{(2.18)}
\sum_{k=0}^p(-1)^{p-k}\binom{p}{k}k^i T^{*k}\mathcal{I}^{m}_{A}(T)T^{p-k}=0,\quad i=0,1,\cdots,p-n.
\end{equation}
\end{enumerate}
\end{proposition}

\begin{proof}
\begin{enumerate}
\item We aim to prove (\ref{(2.17)}) by induction on $p$. Assume that $T$ is $(A,(m, n))$-isosymmetric, then $(2)$-Proposition (\ref{waw}) yields that $T$ is $(A,(m^\prime , n))$-isosymmetric for each  $m^\prime \geq m.$ Thus, for $i=0$ the proof of $(\ref{(2.17)})$ is immediate. Moreover, the result is true for $p=m$. Assume now that $(\ref{(2.17)})$ is true for $i \in \{1,2,\cdots,p-m\}$ and prove it for $i\in \{1,2,\cdots,p+1-m \}$. By the induction hypothesis, we obtain
\begin{eqnarray*}
&&\sum_{k=0}^{p+1}(-1)^{p+1-k}\binom{p+1}{k}k^iT^{*k}\mathcal{S}^{n}_{A}(T)T^{k}\\
&=&\sum_{k=1}^{p+1}(-1)^{p+1-k}\binom{p+1}{k}k^iT^{*k}\mathcal{S}^{n}_{A}(T)T^{k}\\
&=&\sum_{k=0}^{p}(-1)^{p-k}\binom{p+1}{k+1}(k+1)^iT^{*k+1}\mathcal{S}^{n}_{A}(T)T^{k+1}\\
&=&(p+1)T^{*}\bigg\{\sum_{k=0}^{p}(-1)^{p-k}\binom{p}{k}(k+1)^{i-1}T^{*k}\mathcal{S}^{n}_{A}(T)T^{k}\bigg\}T\\
&=&(p+1)T^{*}\bigg\{\sum_{k=0}^{p}(-1)^{p-k}\binom{p}{k}\bigg(\sum_{j=0}^{i-1}\binom{i-1}{j}k^j\bigg)T^{*k}\mathcal{S}^{n}_{A}(T)T^{k}\bigg\}T\\
&=&(p+1)T^{*}\bigg\{\sum_{j=0}^{i-1}\binom{i-1}{j}\bigg(\underbrace{\sum_{k=0}^{p}(-1)^{p-k}\binom{p}{k}k^jT^{*k}\mathcal{S}^{n}_{A}(T)T^{k}}_{=0}\bigg)\bigg\}T= 0.
\end{eqnarray*}
\item The remaining identity (\ref{(2.18)}) in the statement $(2)$ also holds by a similar method.
\end{enumerate}
\end{proof}

Now, we devote most of our interest to distinguish some connection between $(A,(m, n))$-isosymmetries, skew $(A,(m, n))$-isosymmetries and their associated exponential operator. We begin our study by introducing the notion of left $A$-invertible operators which will play a
central role in characterizing skew $(A;(m, n))$-isosymmetries and giving some their interesting properties.

\begin{definition}
Let $A\in \mathcal{B}(H)^+$ and $T,\;S\in \mathcal{B}(H)$. We say that the operator $T$ is left $($resp., right$)$ $(A,m)$-invertible by $S$, for some integer $m \geq1$, if $\big\{(yx-1)^m\big\}_\mathbf{a}\big(S,T,A)=0$ $\big($resp. $\big\{(yx-1)^m\big\}_\mathbf{a}\big(T,S,A)=0\big).$
\end{definition}

\begin{proposition}\label{dr1}
Let $A\in \mathcal{B}(H)^+$ and $R_1,\,R_2,\, S_1,\, S_2 \in \mathcal{B}(H)$  such that $R_1 R_2$ = $R_2 R_1$ and $S_1 S_2$ = $S_2 S_1$.
Assume that $R_1$ is a left $(A,m)$-inverse of $S_1$ and $R_2$ is a left $(A,m)$-inverse of $S_2$, then $R_1R_2$ is a left $(A,m+n-1)$-inverse of $S_1S_2$.
\end{proposition}

\begin{proof}
Fix $x,y\in H$ and let $a_{i,j} =\langle R_{1}^i R_{2}^j AS_{1}^i S_{2}^j \,x\mid y\rangle$. Then,
for all non-negative integers $k$ and $l$, we have
$$a_{k+i,l} =\langle R_{1}^i AS_{1}^i\big(S_{1}^k S_{2}^l \,x\big)\mid \big(R_{1}^{*k} R_{2}^{*l} y\big)\rangle,\quad
a_{k,l+j} =\langle R_{2}^j AS_{2}^j\big(S_{1}^k S_{2}^l \,x\big)\mid \big(R_{1}^{*k} R_{2}^{*l} y\big)\rangle .$$
The left $(A,m)$-invertibility of $S_1$ by $R_1$ and the left $(A,n)$-invertibility of $S_2$ by $R_2$ gives
$$\sum_{i=0}^{m}(-1)^{m-i} a_{k+i,l}=0, \quad \sum_{j=0}^{n}(-1)^{n-j}a_{k,l+j}=0.$$
Applying \cite[Corollary 2.5]{Dudl} we deduce that
$$0=\sum_{s=0}^{m+n-1}(-1)^{m+n-1-s} a_{s,s}=\sum_{s=0}^{m+n-1}(-1)^{m+n-1-s} \;{m+n-1\choose s}\;\big\langle(R_{1}R_{2})^{s}A( S_{1}S_{2})^{s})x|y\big\rangle$$
Since $x$ and $y$ are arbitrary in $H$, it holds
$$\sum_{s=0}^{m+n-1}(-1)^{m+n-1-s} \;{m+n-1\choose s}\;(R_{1}R_{2})^{s}A( S_{1}S_{2})^{s})=0.$$
Hence, $R_1R_2$ is a left $(A,m+n-1)$-inverse of $S_1S_2.$

\end{proof}

\begin{proposition}\label{lft1r}
Let $A\in \mathcal{B}(H)^+$ and $T\in \mathcal{B}(H)$. Then, the following statements hold.
\begin{enumerate}
\item $T$ is $(A,(m, n))$-isosymmetric if and only if, for each positive integer $k$, it holds
\begin{equation}\label{lft1.0.0}
\Big(e^{isT}\Big)^{*k}\mathcal{I}^{m}_{A}(T)\Big(e^{isT}\Big)^{k}=\sum_{h=0}^{n-1}\frac{(-isk)^{h}}{h!}\;\Omega_{A}^{m,h}(T).
\end{equation}
\item $T$ is skew $(A,(m, n))$-isosymmetric if and only if, for each positive integer $k$, it holds
\begin{equation}\label{lft1.0}
\Big(e^{isT^*}\Big)^{k}\mathcal{I}^{m}_{A}(T)\Big(e^{isT}\Big)^{k}=\sum_{h=0}^{n-1}\frac{(isk)^{h}}{h!}\Lambda_{A}^{m,h}(T).
\end{equation}
\end{enumerate}
\end{proposition}

\begin{proof}
\begin{enumerate}
\item W have shown, in \cite{a.1}, that
\begin{eqnarray}\label{s11}
(e^{isT})^*A(e^{isT})&=&\mathcal{S}^{0}_{A}(T)+(-is)\mathcal{S}^{1}_{A}(T)+\frac{(-is)^2}{2!}\mathcal{S}^{2}_{A}(T)+\frac{(-is)^3}{3!}\mathcal{S}^{3}_{A}(T)+\cdots .
\end{eqnarray}
It holds from (\ref{s11}) that
\begin{eqnarray*}\label{bnl}
&&\Big(e^{isT}\Big)^{*k}\mathcal{I}^{m}_{A}(T)\Big(e^{isT}\Big)^{k}=\sum_{j=0}^{m}(-1)^{m-j} \;{m\choose j}T^{*j}\Big(e^{isT}\Big)^{*k} A\Big(e^{isT}\Big)^k T^{j}\nonumber\\
&=& \sum_{j=0}^{m}(-1)^{m-j} \;{m\choose j}T^{*j}\mathcal{S}^{0}_{A}(T)T^{j}+(-is)k\sum_{j=0}^{m}(-1)^{m-j} \;{m\choose j}T^{*j}\mathcal{S}^{1}_{A}(T)T^{j}\nonumber\\
&&+\frac{(-is)^2k^2}{2!}\sum_{j=0}^{m}(-1)^{m-j} \;{m\choose j}T^{*j}\mathcal{S}^{2}_{A}(T)T^{j}+\frac{(-is)^3k^3}{3!}\sum_{j=0}^{m}(-1)^{m-j} \;{m\choose j}T^{*j}\mathcal{S}^{3}_{A}(T)T^{j}+\cdots +\nonumber\\
&&+\cdots+\frac{(-is)^{n-1}k^{n-1}}{(n-1)!}\sum_{j=0}^{m}(-1)^{m-j} \;{m\choose j}T^{*j}\mathcal{S}^{n-1}_{A}(T)T^{j}\nonumber\\
&&+\frac{(-is)^{n}k^{n}}{n!}\sum_{j=0}^{m}(-1)^{m-j} \;{m\choose j}T^{*j}\mathcal{S}^{n}_{A}(T)T^{j}+\cdots \nonumber\\
&=&\Omega_{A}^{m,0}(T)+(-is)k\Omega_{A}^{m,1}(T)+\frac{(-is)^2k^2}{2!}\Omega_{A}^{m,2}(T)+\frac{(-is)^3k^3}{3!}\Omega_{A}^{m,3}(T)+\cdots \nonumber\\
&&+\cdots+\frac{(-is)^{n-1}k^{n-1}}{(n-1)!}\Omega_{A}^{m,n-1}(T)+\frac{(-is)^{n}k^{n}}{n!}\Omega_{A}^{m,n}(T)+\cdots
\end{eqnarray*}
Since $T$ is $(A,(m, n))$-isosymmetric, $\Omega_{A}^{m,n}(T)=0$ for $n^\prime\geq n$. This completes the proof.
\item A simple calculation shows that
\begin{eqnarray}\label{s1}
e^{isT^*}Ae^{isT}&=&\zeta^{0}_{T}(A)+(is)\zeta^{1}_{T}(A)+\frac{(is)^2}{2!}\zeta^{2}_{T}(A)+\frac{(is)^3}{3!}\zeta^{3}_{T}(A)+\cdots .
\end{eqnarray}
It holds from (\ref{s1}) that
\begin{eqnarray*}\label{bnl}
&&\Big(e^{isT^*}\Big)^{k}\mathcal{I}^{m}_{A}(T)\Big(e^{isT}\Big)^{k}=\sum_{j=0}^{m}(-1)^{m-j} \;{m\choose j}T^{*j}\Big(e^{isT^*}\Big)^{k} A\Big(e^{isT}\Big)^k T^{j}\nonumber\\
&=& \sum_{j=0}^{m}(-1)^{m-j} \;{m\choose j}T^{*j}AT^{j}+(is)k\sum_{j=0}^{m}(-1)^{m-j} \;{m\choose j}T^{*j}\zeta^{1}_{A}(T)T^{j}\nonumber\\
&&+\frac{(is)^2k^2}{2!}\sum_{j=0}^{m}(-1)^{m-j} \;{m\choose j}T^{*j}\zeta^{2}_{A}(T)T^{j}+\frac{(is)^3k^3}{3!}\sum_{j=0}^{m}(-1)^{m-j} \;{m\choose j}T^{*j}\zeta^{3}_{A}(T)T^{j}+\cdots +\nonumber\\
&&+\cdots+\frac{(is)^{n-1}k^{n-1}}{(n-1)!}\sum_{j=0}^{m}(-1)^{m-j} \;{m\choose j}T^{*j}\zeta^{n-1}_{A}(T)T^{j}\nonumber\\
&&+\frac{(is)^{n}k^{n}}{n!}\sum_{j=0}^{m}(-1)^{m-j} \;{m\choose j}T^{*j}\zeta^{n}_{A}(T)T^{j}+\cdots \nonumber\\
&=&\Lambda_{A}^{m,0}(T)+(is)k\Lambda_{A}^{m,1}(T)+\frac{(is)^2k^2}{2!}\Lambda_{A}^{m,2}(T)+\frac{(is)^3k^3}{3!}\Lambda_{A}^{m,3}(T)+\cdots \nonumber\\
&&+\cdots+\frac{(is)^{n-1}k^{n-1}}{(n-1)!}\Lambda_{A}^{m,n-1}(T)+\frac{(is)^{n}k^{n}}{n!}\Lambda_{A}^{m,n}(T)+\cdots
\end{eqnarray*}
\noindent Since $T$ is skew $(A,(m, n))$-isosymmetric, $\Lambda_{A}^{m,n}(T)=0$ for $n^\prime\geq n$, and this allows to conclude.
\end{enumerate}
\end{proof}

\begin{theorem}\label{leftinv}
Let $A\in \mathcal{B}(H)^+$ and $T\in \mathcal{B}(H)$. Then the following statements hold.
\begin{enumerate}
\item If $T$ is $(A,(m, n))$-isosymmetric, then for every $s\in \mathbb{R}$, it holds
\begin{eqnarray}
 \sum_{k=0}^{n}(-1)^{n-k} \;{n\choose k}\;\Big(e^{isT}\Big)^{*k} \mathcal{I}^{m}_{A}(T)\Big(e^{isT}\Big)^k &=&0.
\end{eqnarray}
\item If $T$ is skew $(A,(m, n))$-isosymmetric, then for every $s\in \mathbb{R}$, it holds
\begin{eqnarray*}
\sum_{k=0}^{n}(-1)^{n-k} \;{n\choose k}\;\Big(e^{isT^*}\Big)^{k} \mathcal{I}^{m}_{A}(T)\Big(e^{isT}\Big)^k&=&0.
\end{eqnarray*}
\end{enumerate}
\end{theorem}

\begin{proof}
\begin{enumerate}
\item By using (\ref{lft1.0}) and taking into account the identity
\begin{equation}\label{fcv}
\dd\sum_{k=0}^{n}(-1)^{n-k} \;{n\choose k}\;k^{j}=
\left\{
\begin{array}{ll}
0, & \hbox{if $0\leq j\leq n-1$;} \\
n!, & \hbox{if $j=n$,}
\end{array}
\right.
\end{equation}
 it follows
\begin{eqnarray*}
&& \sum_{k=0}^{n}(-1)^{n-k} \;{n\choose k}\;\Big(e^{isT}\Big)^{*k} \mathcal{I}^{m}_{A}(T)\Big(e^{isT}\Big)^k \\
&=& \sum_{k=0}^{n}(-1)^{n-k} \;{n\choose k}\;\bigg\{\Omega_{A}^{m,0}(T)+(-is)k\Omega_{A}^{m,1}(T)+\frac{(-is)^2k^2}{2!}\Omega_{A}^{m,2}(T)+\frac{(-is)^3k^3}{3!}\Omega_{A}^{m,3}(T)+\cdots \\
&&+\cdots+\frac{(-is)^{n-1}k^{n-1}}{(n-1)!}\Omega_{A}^{m,n-1}(T)+\frac{(-is)^{n}k^{n}}{n!}\Omega_{A}^{m,n}(T)+\cdots\bigg\}\\
&=&\Omega_{A}^{m,0}(T)\bigg(\sum_{k=0}^{n}(-1)^{n-k} \;{n\choose k}\bigg)+(-is)\Omega_{A}^{m,1}(T)\bigg(\sum_{k=0}^{n}(-1)^{n-k} \;{n\choose k}\;k\bigg)\\
&&+\frac{(-is)^2}{2!}\Omega_{A}^{m,2}(T)\bigg(\sum_{k=0}^{n}(-1)^{n-k} \;{n\choose k}\;k^2\bigg)+\cdots+\\
&&+\cdots+\frac{(-is)^{n-1}}{(n-1)!}\Omega_{A}^{m,n-1}(T)\bigg(\sum_{k=0}^{n}(-1)^{n-k} \;{n\choose k}\;k^{n-1}\bigg)=0.
\end{eqnarray*}
\item By using (\ref{lft1.0}) and (\ref{fcv}), it follows
\begin{eqnarray*}
&& \sum_{k=0}^{n}(-1)^{n-k} \;{n\choose k}\;\Big(e^{isT^*}\Big)^{k} \mathcal{I}^{m}_{A}(T)\Big(e^{isT}\Big)^k \\
&=& \sum_{k=0}^{n}(-1)^{n-k} \;{n\choose k}\;\bigg\{\Lambda_{A}^{m,0}(T)+(is)k\Lambda_{A}^{m,1}(T)+\frac{(is)^2k^2}{2!}\Lambda_{A}^{m,2}(T)+\frac{(is)^3k^3}{3!}\Lambda_{A}^{m,3}(T)+\cdots \\
&&+\cdots+\frac{(is)^{n-1}k^{n-1}}{(n-1)!}\Lambda_{A}^{m,n-1}(T)+\frac{(is)^{n}k^{n}}{n!}\Lambda_{A}^{m,n}(T)+\cdots\bigg\}\\
&=&\Lambda_{A}^{m,0}(T)\bigg(\sum_{k=0}^{n}(-1)^{n-k} \;{n\choose k}\bigg)+(is)\Lambda_{A}^{m,1}(T)\bigg(\sum_{k=0}^{n}(-1)^{n-k} \;{n\choose k}\;k\bigg)\\
&&+\frac{(is)^2}{2!}\Lambda_{A}^{m,2}(T)\bigg(\sum_{k=0}^{n}(-1)^{n-k} \;{n\choose k}\;k^2\bigg)+\cdots+\\
&&+\cdots+\frac{(is)^{n-1}}{(n-1)!}\Lambda_{A}^{m,n-1}(T)\bigg(\sum_{k=0}^{n}(-1)^{n-k} \;{n\choose k}\;k^{n-1}\bigg)=0.
\end{eqnarray*}
\end{enumerate}
\end{proof}

The following corollary is an immediate consequence of Theorem \ref{leftinv}.
\begin{corollary}\label{cor3.4}
Let $A\in \mathcal{B}(H)^+$ and $T\in \mathcal{B}(H)$. Assume that $\mathcal{I}^{m}_{A}(T)\geq0$. Then, we have:
\begin{enumerate}
\item If $T$ is $(A,(m, n))$-isosymmetric, then for every $s\in \mathbb{R}$, the operator $e^{isT}$
 is $(\mathcal{I}^{m}_{A}(T),n)$-isometric. In particular, if $T$ is $(A,n)$-symmetric, then $e^{isT}$ is $(A,n)$-isometric.
\item If $T$ is skew $(A,(m, n))$-isosymmetric, then for every $s\in \mathbb{R}$, the operator $e^{isT}$
 is left $(\mathcal{I}^{m}_{A}(T),n)$-invertible. In particular, if $T$ is skew $(A,n)$-symmetric, then $T$ is left $(A,n)$-invertible with a left inverse $e^{isT^*}$.
\end{enumerate}
\end{corollary}

\begin{theorem}\label{theo2.8}
Let $A\in \mathcal{B}(H)^+$ and $R,\,S \in \mathcal{B}(H)$ such that $RS=SR$. If $R$ is skew $(A,m)$-symmetric and $S$ is skew $(A,n)$-symmetric, then $e^{is(R+S)}$ is left $(A,m+n-1)$-invertible with left inverse $e^{is(R+S)^*}$.
\end{theorem}

\begin{proof}
Since $R$ is skew $(A,m)$-symmetric and $S$ is skew $(A,n)$-symmetric, it holds from Corollary \ref{cor3.4} that $S_1 =e^{isR}$ is is left $(A,m)$-invertible with left inverse $R_1=e^{isR^*}$ and $S_2 =e^{isS}$ is left $(A,n)$-invertible with left inverse $R_2=e^{isS^*},$ for every $s\in \mathbb{R}$. It is clear that $R_1R_2 = R_2R_1$ and $S_1S_2 = S_2S_1$. Applying Proposition \ref{dr1}, we deduce that $S_1S_2=e^{is(S+R)}$ is left $(A,m+n-1 )$-invertible with left inverse $R_1R_2=e^{is(R+S)^*}.$

\end{proof}

In \cite{a.1}, we have introduced the following algebraic condition on $T\in \mathcal{B}(H)$
\begin{equation}\label{rf}
\mathbf{\big( POL_{n}(A)\big)}:\quad (e^{isT })^*A\,(e^{isT})=A+\sum^{n}_{l=1}s^{l}B_{l},\quad B_{l}:=B_{l}(A,T ,T^*)\in \mathcal{B}(H).
\end{equation}
 Note that an induction argument gives
\begin{equation*}
\frac{d^{n+1}}{ds^{n+1}}(e^{isT })^*A\,(e^{isT})=(-i)^{n+1}(e^{isT })^*\,\mathcal{S}^{n+1}_{A}(T)\;(e^{isT}).
\end{equation*}
The $\mathbf{\big( POL_{n}(A)\big)}$ condition is equivalent to $\frac{d^{n+1}}{ds^{n+1}}(e^{isT })^*A\,(e^{isT})=0$ and consequently to $\mathcal{S}^{n+1}_{A}(T)=0$. Hence, $T\in \mathcal{B}(\mathbb{H})$ is $(A,n)$-symmetric if it satisfies $\mathbf{\big( POL_{n-1}(A)\big)}.$\\

\noindent Inspired by (\ref{rf}), we introduce
\begin{equation}\label{vv.0}
\mathbf{\big(IPOL_{n}(A)\big)}:\quad (e^{isT })^*\mathcal{I}^{m}_{A}(T)\,(e^{isT})=\Omega_{A}^{m,0}(T)+\sum^{n}_{l=1}s^{l}C_{l},\quad C_{l}:=C_{l}(A,T ,T^*)\in \mathcal{B}(H).
\end{equation}
For $m\in \mathbb{N}$, we obtain
\begin{equation*}
\frac{d^{n+1}}{ds^{n+1}}(e^{isT })^*\mathcal{I}^{m}_{A}(T)\,(e^{isT})=(-i)^{n+1}(e^{isT })^*\,\Omega_{A}^{m,n+1}(T)\;(e^{isT}).
\end{equation*}
Moreover,
\begin{eqnarray}\label{amie}
\Big(e^{isT}\Big)^*\mathcal{I}^{m}_{A}(T)\Big(e^{isT}\Big)
&=&\sum_{h\geq0}\frac{s^{h}}{h!}\bigg(\frac{d^{h}}{ds^{h}}(e^{isT })^*\mathcal{I}^{m}_{A}(T)\,(e^{isT})_{\big|_{s=0}}\bigg) \nonumber\\
&=&\sum_{h\geq0}\frac{(-is)^{h}}{h!}\;\Omega_{A}^{m,h}(T)\nonumber.
\end{eqnarray}
Referring to Proposition \ref{lft1r} and applying the identity (\ref{lft1.0.0}) for $k=1$, we deduce that $\mathbf{\big(IPOL_{n}(A)\big)}$ is equivalent to $\frac{d^{n+1}}{ds^{n+1}}(e^{isT })^*\mathcal{I}^{m}_{A}(T)\,(e^{isT})=0$ and consequently to $\Omega_{A}^{m,n+1}(T)=0$.

\begin{remark}For $A\in \mathcal{B}(H)^+$ and $T\in \mathcal{B}(H)$, $T$ is $(A,(m, n))$-isosymmetric if and only if $T$ satisfies $\mathbf{\big( IPOL_{n-1}(A)\big)}.$
\end{remark}

\begin{theorem}\label{n-1}
Let $A\in \mathcal{B}(H)^+$ and $T\in \mathcal{B}(H)$. Then the following properties hold.
\begin{enumerate}
\item If $m$ is even, then every invertible $(A;(m, n))$-isosymmetric operator $T$ satisfying $\mathcal{S}^{n}_{A}(T)\geq0$, is $(A;(m-1, n))$-isosymmetric.
\item If $n$ is even, then every $(A,(m, n))$-isosymmetric operator $T$ satisfying $\mathcal{I}^{m}_{A}(T)\geq0$ is $(A,(m, n-1))$-isosymmetric.
\end{enumerate}
\end{theorem}

\begin{proof}
\begin{enumerate}
\item If $\mathcal{S}^{n}_{A}(T)\geq0$ for some positive integer $n$, then $T$ is $(A;(m, n))$-isosymmetric
if and only if $T$ is $(\mathcal{S}^{n}_{A}(T),m)$-isometric. Note that if $T$ is invertible, then $\mathcal{S}^{n}_{A}(T^{-1})\geq0$ and, consequently, $T^{-1}$ is $(A;(m, n))$-isosymmetric if and only if $T^{-1}$ is $(\mathcal{S}^{n}_{A}(T^{-1}),m)$-isometric. It follows from that $\Omega_{A}^{m-1,n}(T)\geq0$ and $\Omega_{A}^{m-1,n}(T^{-1})\geq0.$ Hence, the following holds
\begin{equation}\label{(2.17.00)}
0\leq(T^{-1})^{*m-1+n}\Omega_{A}^{m-1,n}(T)(T^{-1})^{m-1+n}=\left\{
  \begin{array}{ll}
    \Omega_{T^{-1}}^{m-1,n}(A), & \hbox{if $m-1$ is even;} \\
    -\Omega_{T^{-1}}^{m-1,n}(A), & \hbox{if $m-1$ is odd.}
  \end{array}
\right.
\end{equation}
Since $m$ is an even number, $m-1$ is odd and (\ref{(2.17.00)}) gives $\Omega_{A}^{m-1,n}(T)\leq0$. Therefore, we have $\Omega_{A}^{m-1,n}(T)=0.$ Hence, $T$ is $(A;(m-1, n))$-isosymmetric.
\item Since $T$ is $(A,(m, n))$-isosymmetric, it holds for $n$ even $($that is $n=2k)$,
$$\Big(e^{isT}\Big)^*\mathcal{I}^{m}_{A}(T)\Big(e^{isT}\Big)=\sum_{h=0}^{2k-1}\frac{(-is)^{h}}{h!}\Omega_{A}^{m,h}(T).$$
The positivity of $\mathcal{I}^{m}_{A}(T)$ implies
\begin{eqnarray*}
0&\leq&\frac{\Big(e^{isT}\Big)^*\mathcal{I}^{m}_{A}(T)\Big(e^{isT}\Big)}{s^{2k-1}}=\frac{1}{s^{2k-1}}\Omega_{A}^{m,0}(T)+\sum_{h=0}^{2k-2}\frac{(-is)^{h}}{h!}
\Omega_{A}^{m,h}(T)\\
&&+\frac{(-i)^{2k-1}}{(2k-1)!}\Omega_{A}^{m,2k-1}(T)\longrightarrow \frac{(-i)^{2k-1}}{(2k-1)!}\Omega_{A}^{m,2k-1}(T) \qquad(s\longrightarrow+\infty).
\end{eqnarray*}
On the other hand, we have
\begin{eqnarray*}
0&\leq&\frac{-\Big(e^{isT}\Big)^*\mathcal{I}^{m}_{A}(T)\Big(e^{isT}\Big)}{s^{2k-1}}=\frac{1}{s^{2k-1}}\big(-\Omega_{A}^{m,0}(T)\big)+\sum_{h=0}^{2k-2}\frac{(-is)^{h}}{h!}
\big(-\Omega_{A}^{m,h}(T)\big)\\
&&+\frac{(-i)^{2k-1}}{(2k-1)!}\big(-\Omega_{A}^{m,2k-1}(T)\big)\longrightarrow \frac{(-i)^{2k-1}}{(2k-1)!}\big(-\Omega_{A}^{m,2k-1}(T)\big) \qquad(s\longrightarrow-\infty).
\end{eqnarray*}
Hence, $\Omega_{A}^{m,2k-1}(T)=0$, which implies that $T$ is $(A,(m, 2k-1))$-isosymmetric.
\end{enumerate}
\end{proof}

\noindent The following corollary follows immediately from Theorem \ref{n-1} and summarizes few known facts.
\begin{corollary}\label{ddf}
Let $A\in \mathcal{B}(H)^+$ and $T\in \mathcal{B}(H)$. Then, the following properties hold.
\begin{enumerate}
\item If $m$ is even, then every invertible $(A,m)$-isometric operator is $(A,m-1)$-isometric.
\item If $n$ is even, then every $(A, n)$-symmetric operator is $(A, n-1)$-symmetric.
\end{enumerate}
\end{corollary}

\noindent Let consider, now,
\begin{equation}\label{rfklm}
\widetilde{\mathbf{\big(POL_{n}(A)\big)}}:\quad (e^{isT^* })A\,(e^{isT})=A+\sum^{n}_{l=1}s^{l}D_{l},\quad D_{l}:=D_{l}(A,T ,T^*)\in \mathcal{B}(H).
\end{equation}
Note that
\begin{equation*}
\frac{d^{n}}{ds^{n}}(e^{isT^* })A\,(e^{isT})=i^{n}e^{isT^*}\,\zeta^{n}_{T}(A)\;e^{isT},\quad n\in \mathbb{N}.
\end{equation*}
The $\widetilde{\mathbf{\big(POL_{n}(A)\big)}}$ condition is equivalent to $\frac{d^{n+1}}{ds^{n+1}}\big(e^{isT^*}A\,e^{isT}\big)=0$ and consequently to $\zeta^{n+1}_{T}(A)=0$. Hence, $T\in \mathcal{B}(H)$ is skew $(A,n)$-symmetric if it satisfies $\widetilde{\mathbf{\big(POL_{n-1}(A)\big)}}.$\\

\noindent Inspired by (\ref{rfklm}), we introduce
\begin{equation}\label{Irfklm}
\widetilde{\mathbf{\big(IPOL_{n}(A)\big)}}:\quad (e^{isT^* })\mathcal{I}^{m}_{A}(T)\,(e^{isT})=\Lambda_{A}^{m,0}(T)+\sum^{n}_{l=1}s^{l}E_{l},\quad E_{l}:=E_{l}(A,T ,T^*)\in \mathcal{B}(H).
\end{equation}
Referring to Proposition \ref{lft1r}, $T$ is skew $(A,(m, n))$-isosymmetric if and only if
\begin{eqnarray*}
\Big(e^{isT^*}\Big)\mathcal{I}^{m}_{A}(T)\Big(e^{isT}\Big)&=&\sum_{h=0}^{n-1}\frac{s^{h}}{h!}\bigg(\frac{d^{h}}{ds^{h}}(e^{isT^* })\mathcal{I}^{m}_{A}(T)\,(e^{isT})_{\big|_{s=0}}\bigg)\\
&=&\sum_{h=0}^{n-1}\frac{(is)^{h}}{h!}\Lambda_{A}^{m,h}(T),
\end{eqnarray*}
that is $T$ satisfies $\widetilde{\mathbf{\big(IPOL_{n-1}(A)\big)}}.$\\

\noindent In Corollary \ref{ddf} (\cite[Theorem 2.5.]{a.2}), we have shown that if $T$ is $(A,n)$-symmetric and $n$ is even, then $T$ is always $(A,n-1)$-symmetric. Under suitable conditions, we extend this property to skew $(A,m)$-symmetric operators, and more generally to skew $(A,(m, n))$-isosymmetries. For $A\in \mathcal{B}(H)^+$ and $T \in \mathcal{B}(H)$, let consider the following hypothesis:
\begin{enumerate}
\item $(H_1):\dd\sum_{k=0}^{n-1}(-1)^{n-1-k} \;{n-1\choose k}\;\Big(e^{isT^*}\Big)^k\mathcal{I}^{m}_{A}(T)\Big(e^{isT}\Big)^k \geq0.$
\item $(H_2):\dd\sum_{k=0}^{n-1}(-1)^{n-1-k} {n-1\choose k}\Big(e^{isT^*}\Big)^{n-1-k}\mathcal{I}^{m}_{A}(T)\Big(e^{isT}\Big)^{n-1-k}\geq0.$
\end{enumerate}
\begin{proposition}\label{powx}
Let $A\in \mathcal{B}(H)^+$ and $T \in \mathcal{B}(H)$. Assume $(H_1)$ and $(H_2)$ are fulfilled and
$$\sum_{k=0}^{n}(-1)^{n-k} \;{n\choose k}\;\Big(e^{isT^*}\Big)^{k} \mathcal{I}^{m}_{A}(T)\Big(e^{isT}\Big)^k=0.$$
If $n$ is even, then it holds
$$\sum_{k=0}^{n-1}(-1)^{n-1-k} \;{n-1\choose k}\;\Big(e^{isT^*}\Big)^{k} \mathcal{I}^{m}_{A}(T)\Big(e^{isT}\Big)^k=0.$$
\end{proposition}

\begin{proof}Since $n$ is even, we have $(-1)^k = -(-1)^{n-1-k}$. Then, we obtain
\begin{eqnarray*}
&&\Big(e^{isT^*}\Big)^{n-1}\bigg\{\sum_{k=0}^{n-1}(-1)^{n-1-k} \;{n-1\choose k}\;\Big(e^{-isT^*}\Big)^{k} \mathcal{I}^{m}_{A}(T)\Big(e^{-isT}\Big)^k\bigg\}\Big(e^{isT}\Big)^{n-1}\\
&=&\sum_{k=0}^{n-1}(-1)^{n-1-k} \;{n-1\choose k}\;\Big(e^{isT^*}\Big)^{n-1-k} \mathcal{I}^{m}_{A}(T) \Big(e^{isT}\Big)^{n-1-k}\\
&=&-\sum_{k=0}^{n-1}(-1)^{k} \;{n-1\choose k}\;\Big(e^{isT^*}\Big)^{n-1-k} \mathcal{I}^{m}_{A}(T) \Big(e^{isT}\Big)^{n-1-k}\geq0.
\end{eqnarray*}
This completes the proof.

\end{proof}

\begin{theorem}\label{theo2.7}
Let $A\in \mathcal{B}(H)^+$ and $T \in \mathcal{B}(H)$. Assume $(H_1)$ and $(H_2)$ are fulfilled. If $T$ is skew $(A,(m, n))$-isosymmetric where $n$ is even, then $T$ is skew $(A,(m, n-1))$-isosymmetric.
\end{theorem}

\begin{proof}Since $T$ is skew $(A,(m, n))$-isosymmetric, $\Lambda_{A}^{m,p}(T)=0$ for all $p\geq n$ and, by Theorem \ref{leftinv}, the operator $e^{isT}$ satisfies
$$\sum_{k=0}^{n}(-1)^{n-k} \;{n\choose k}\;\Big(e^{isT^*}\Big)^{k} \mathcal{I}^{m}_{A}(T)\Big(e^{isT}\Big)^k=0.$$
On the other hand, since $n$ is even, it follows from Proposition \ref{powx}
\begin{eqnarray*}
0&=& \sum_{k=0}^{n-1}(-1)^{n-1-k} \;{n-1\choose k}\;\Big(e^{isT^*}\Big)^{k} \mathcal{I}^{m}_{A}(T)\Big(e^{isT}\Big)^k \\
 &=&\sum_{k=0}^{n-1}(-1)^{n-1-k} \;{n-1\choose k}\bigg\{\Lambda_{A}^{m,0}(T)+(is)(n-1-k)\Lambda_{A}^{m,1}(T)+\frac{(is)^2(n-1-k)^2}{2!}\Lambda_{A}^{m,2}(T)\\
&&+\frac{(is)^3(n-1-k)^3}{3!}\Lambda_{A}^{m,3}(T)+\cdots +\frac{(is)^{n-1}(n-1-k)^{n-1}}{(n-1)!}\Lambda_{A}^{m,n-1}(T)\bigg\}\\
&=&\Lambda_{A}^{m,0}(T)\bigg(\sum_{k=0}^{n-1}(-1)^{n-1-k} \;{n-1\choose k}\bigg)+(is)\Lambda_{A}^{m,1}(T)\bigg(\sum_{k=0}^{n-1}(-1)^{n-1-k} \;{n-1\choose k}(n-1-k)\bigg)\\
&&+\frac{(is)^2}{2!}\Lambda_{A}^{m,2}(T)\bigg(\sum_{k=0}^{n-1}(-1)^{n-1-k} \;{n-1\choose k}(n-1-k)^2\bigg)+\cdots+\\
&&+\cdots+\frac{(is)^{n-1}}{(n-1)!}\Lambda_{A}^{m,n-1}(T)\bigg(\sum_{k=0}^{n-1}(-1)^{n-1-k} \;{n-1\choose k}(n-1-k)^{n-1}\bigg).
\end{eqnarray*}
Using (\ref{fcv}), we obtain $\Lambda_{A}^{m,n-1}(T)=0$. Hence $T$ is skew $(A,(m, n-1))$-isosymmetric.

\end{proof}

In the rest of this section, we investigate some spectral properties of $(A,(m, n))$-isosymmetric operators. The showed results extends those of \cite{a.2}, \cite{sjn}, \cite{a.1}, \cite{mark} and \cite{X27}.
\begin{theorem}
Let $A\in \mathcal{B}(H)^+$, $T \in \mathcal{B}(H)$. If $T$ is an $(A,(m, n))$-isosymmetric operator and $0\not \in \sigma_{ap}(A)$, then the following statements hold.
\begin{enumerate}
\item $\partial{\sigma(T)} \subset \partial{\textbf{D}}\cup \mathbb{R}$.
\item Let $x, y$ be unit vectors and $x_k, y_k$ be sequences of unit vectors of $H$. Then, we have
\begin{enumerate}
\item If $Tx = \lambda x$, $Ty = \mu y$, $\lambda \neq \mu$ and $\lambda \neq \overline{\mu}$, then $\langle A x\,|\, y\rangle = 0$.
\item If $(T - \lambda)x_k \rightarrow 0$, $(T - \mu)y_k \rightarrow 0$ $(k \rightarrow \infty,\,\lambda \neq \mu,\,\lambda \neq \overline{\mu})$, then $\dd\lim_{k\rightarrow\infty}\langle A x_k\,|\, y_k\rangle = 0.$
\end{enumerate}
\end{enumerate}
\end{theorem}

\begin{proof}
\begin{enumerate}
\item If $\lambda  \in \mathbb{C}$ is in the approximate point spectrum of $T$, then there exists a sequence $(x_j) \subset H$
such that for all $j$, $\Vert x_j \Vert = 1$, and $(T - \lambda) x_j \longrightarrow 0$ as $j \longrightarrow \infty$. Thus, for each integer $i$, $\lim_{j \longrightarrow \infty} (T^i - {\lambda}^i) x_j \longrightarrow 0$. Moreover,
\begin{eqnarray*}
0&=&
\sum_{k=0}^{m}(-1)^{m-k} \;{m\choose k}\;\big\langle\mathcal{S}^{n}_{A}(T)T^{k}x\,\big|\,T^{k}y\big\rangle\\
&=&{\big(\lambda\overline{\lambda} -1\big)}^{m}\big({\big(\lambda - \overline{\lambda}\big)}^{n} \big) \lim_{j \to \infty} \langle  A  x_j \mid x_j \rangle\\
&=& {\big(|\lambda|^2 -1\big)}^{m}{\big(2\Im m\lambda\big)}^{n} \lim_{j \to \infty} \| x_j \|^{2}_{A},
\end{eqnarray*}
which implies that either $|\lambda|=1$ or $\Im m\lambda = 0$ and so, either $\lambda\in \textbf{D}$ or $\lambda\in \mathbb{R}$. Hence, $\sigma_{ap}(T) \subset \partial{\textbf{D}}\cup \mathbb{R}$. Moreover, $\partial{\sigma(T)}  \subset \sigma_{ap}(T) $. So, $\partial{\sigma(T)} \subset \partial{\textbf{D}}\cup \mathbb{R}$.
\item
\begin{enumerate}
\item Let $\lambda$ and $\mu$ be two distinct eigenvalues of $T$ and suppose that $T x = \lambda x$ and $Ty = \mu y$. Then, it holds
\begin{eqnarray*}
0&=&\langle\Omega_{A}^{m,n}(T)x\,|\,y\rangle\\
&=&\bigg(\sum_{k=0}^{m}(-1)^{m-k} \;{m\choose k}\;\lambda^k \overline{\mu}^k\bigg)\bigg(\sum_{i=0}^n(-1)^{n-i} \;{n\choose i}\;\lambda^{n-i} \overline{\mu}^i \bigg)\big\langle Ax\,\big|\,y\big\rangle\\
&=&(\lambda\overline{\mu}-1)^m (\lambda-\overline{\mu})^n\big\langle Ax\,\big|\,y\big\rangle.
\end{eqnarray*}
Since $\lambda \neq \mu$ and $\lambda \neq \overline{\mu}$, we obtain $\big\langle Ax\,\big|\,y\big\rangle=0$.
\item By similar arguments as in the proof of $(3)$, it holds
$$0 = (\lambda\overline{\mu}-1)^m (\lambda-\overline{\mu})^n \lim_{k\rightarrow\infty}\langle A x_k\,|\, y_k\rangle.$$
Since $\lambda \neq \mu$ and $\lambda \neq \overline{\mu}$, it follows $\lim_{k\rightarrow\infty}\langle A x_k\,|\, y_k\rangle=0.$
\end{enumerate}
\end{enumerate}
\end{proof}

\begin{corollary}
Let $m,n$ be positive integers, $T\in \mathcal{B}(H)$ and $A\in \mathcal{B}(H)^+$ with $0 \not\in \sigma_{ap}(A)$. Then the following statements holds.
\begin{enumerate}
\item If $T$ is an $(A,m)$-isometry, then $\partial{\sigma(T)} \subset \partial{\textbf{D}}$ and either
$\sigma(T)\subset\partial{\textbf{D}}$ or $\sigma(T)=\textbf{D}$.
\item If $T$ is an $(A,n)$-symmetry, then $\sigma(T) \subset \mathbb{R}$.
\item If $T$ is $(A,m)$-isometric or $(A,n)$-symmetric, then
\begin{enumerate}
\item $\sigma_p {(T)} \subset \sigma_p {(T^{*})} $.
\item $\sigma_{ap}(T) \subset \sigma_{ap}(T^{*})$.
\end{enumerate}
\item If $T$ is an $(A,(m, n))$-isosymmetry and either $\sigma(T)\subset\{e^{it}\in\partial{\textbf{D}}:0< t<\pi \}$ or $\sigma(T)\subset\{e^{it}\in\partial{\textbf{D}}:\pi< t<2\pi \}$, then $T$ is an $(A,m)$-isometry.
\item If $T$ is an $(A,(m, n))$-isosymmetry and either $\sigma(T)\subset\{t\in \mathbb{R}:|t|>1 \}$ or $\sigma(T)\subset\{t\in \mathbb{R}:|t|<1 \}$, then $T$ is an $(A,n)$-symmetry.
\end{enumerate}
\end{corollary}

An operator $T\in \mathcal{B}(H)$ is said to have the single valued extension property at $\lambda$ (abbreviated SVEP at $\lambda$) if for every open set $D$ containing $\lambda$ the only analytic function $f : D\longrightarrow H$ which satisfies the equation
\begin{equation}\label{svep}
(T-\lambda)f(\lambda) = 0\quad on\;D
\end{equation}
is the constant function $f\equiv0$. $T$ has the SVEP if $T$ has the SVEP at every $\lambda\in \mathbb{C}$ (\cite{r01}).

\begin{theorem}Let $m,n$ be positive integers and let $A\in \mathcal{B}(H)^+$ with $0 \not\in \sigma_p(A)$. Then, every $(A,(m, n))$-isosymmetric operator has the SVEP.
\end{theorem}

\begin{proof}
Assume that $f : D\longrightarrow H$ is any analytic function on $D$ such that $(\ref{svep})$ holds true. From $(\ref{svep})$, it follows that $(T^k -\lambda^k)f(\lambda)= 0$ on $D$ for all positive integer $k$.
This implies that
\begin{eqnarray*}
0=\langle\Omega_{A}^{m,n}(T)f(\lambda)\,|\,f(\lambda)\rangle&=&
\sum_{k=0}^{m}(-1)^{m-k} \;{m\choose k}\;\big\langle\mathcal{S}^{n}_{A}(T)T^{k}f(\lambda)\,\big|\,T^{k}f(\lambda)\big\rangle\\
&=&\sum_{k=0}^{m}(-1)^{m-k} \;{m\choose k}\;|\lambda|^{2k}\big\langle\mathcal{S}^{n}_{A}(T)f(\lambda)\,\big|\,f(\lambda)\big\rangle\\
&=&\big(1-|\lambda|^2\big)^{m}\sum_{j=0}^{n}(-1)^{n-j} \;{n\choose j}\;\big\langle A^{\frac{1}{2}}T^{n-j} f(\lambda)\,\big|\,A^{\frac{1}{2}}T^{j}f(\lambda)\big\rangle\\
&=&\big(1-|\lambda|^2\big)^{m}\sum_{j=0}^{n}(-1)^{n-j} \;{n\choose j}\;\lambda^{n-j} \overline{\lambda}^{j}\big\langle A^{\frac{1}{2}}f(\lambda)\,\big|\,A^{\frac{1}{2}}f(\lambda)\big\rangle\\
&=&\big(1-|\lambda|^2\big)^{m}\big(2\Im m \,\lambda\big)^n\big\|A^{\frac{1}{2}}f(\lambda)\|^2 .
\end{eqnarray*}
This completes the proof.

\end{proof}

\section{ Sum with a nilpotent operator}
\label{S4}
The paper \cite{t2} proves that, for positive integers $m,n$, $(T +Q)$ is $(m+2r -2, n+2r -1)$-isosymmetric if $T$ is an $(m, n)$-isosymmetry, $Q$ is a nilpotent operator with order $r$, and $Q$ doubly commutes with $T$ (i.e $TQ = QT$ and $TQ^* = Q^* T$). The authors, in \cite{t2}, shows this property for $(m, n,C)$-isosymmetric operators, with $C$ a conjugation on a Hilbert space and $T$ and $S$ are $C$-doubly commuting (i.e $TQ = QT$ and $CTCQ^* = Q^*CTC$). Theorem \ref{bibb} extend this result to the sum of an $(A,(m, n))$-isosymmetry and a nilpotent operator.

\begin{lemma}
Let $T,S \in \mathcal{B}(H)$ and $A\in \mathcal{B}(H)^+$. If $T$ and $S$ are doubly commuting, then the following two operator identities hold
\begin{equation}\label{xx1}
\Omega_{A}^{m,n}(T+S)=\sum_{k=0}^{n}\sum_{j=0}^{n-k}\sum_{i+l+h=m}(-1)^k{n\choose k}\;{n-k\choose j}{m\choose i,l,h}\;(T^*+S^*)^i S^{*l+j} \Omega_{A}^{h,n-k-j}(T)T^l S^{i+k}.
\end{equation}
\begin{equation}\label{xx2}
\Lambda_{A}^{m,n}(T+S)=\sum_{k=0}^{n}\sum_{j=0}^{n-k}\sum_{i+l+h=m}{n\choose k}\;{n-k\choose j}{m\choose i,l,h}\;(T^*+S^*)^i S^{*l+j} \Lambda_{A}^{h,n-k-j}(T)T^l S^{i+k}.
\end{equation}
\end{lemma}

\begin{proof} The multinomial theorem gives
\begin{eqnarray*}
&&\big((y_1 +y_2) (x_1 +x_2)-1\big)^m\big\{\big((y_1 +y_2) -(x_1 +x_2)\big)^n\big\}_\mathbf{a}\\
&=&\big((y_1x_1-1)+(y_1 +y_2)x_2+y_2 x_1\big)^m\big\{((y_1-x_1)+(y_2-x_2))^n\big\}_\mathbf{a}\\
&=&\bigg(\sum_{i+l+h=m}{m\choose i,l,h}(y_1 +y_2)^i y^{l}_2 (y_1x_1-1)^h x^{l}_1 x^{i}_2\bigg)\\
&&\cdot\bigg\{\sum_{k=0}^{n}\sum_{j=0}^{n-k}(-1)^k{n\choose k}\;{n-k\choose j}\;y_{2}^{j}(y_1-x_1)^{(n-k-j)}x_{2}^{k}\bigg\}_\mathbf{a}\\
&=&\bigg(\sum_{i+l+h=m}(y_1 +y_2)^i y^{l}_2 (y_1x_1-1)^h x^{l}_1 x^{i}_2 \bigg)\\
&&\cdot\bigg(\sum_{k=0}^{n}\sum_{j=0}^{n-k}(-1)^k{n\choose k}\;{n-k\choose j}\;y_{2}^{j}\Big\{(y_1-x_1)^{(n-k-j)}\Big\}_\mathbf{a}x_{2}^{k}\bigg)\\
&=&\sum_{k=0}^{n}\sum_{j=0}^{n-k}\sum_{i+l+h=m}(-1)^k{n\choose k}\;{n-k\choose j}{m\choose i,l,h}\;(y_1 +y_2)^i y^{l+j}_2 (y_1x_1-1)^h \Big\{(y_1-x_1)^{(n-k-j)}\Big\}_\mathbf{a}x^{l}_1 x^{i+k}_2 .
\end{eqnarray*}
On the other hand, we have
\begin{eqnarray*}
&&\big((y_1 +y_2) (x_1 +x_2)-1\big)^m\big\{\big((y_1 +y_2) -(x_1 +x_2)\big)^n\big\}_\mathbf{a}\\
&=&\big((y_1x_1-1)+(y_1 +y_2)x_2+y_2 x_1\big)^m\big\{((y_1+x_1)+y_2+x_2)^n\big\}_\mathbf{a}\\
&=&\sum_{k=0}^{n}\sum_{j=0}^{n-k}\sum_{i+l+h=m}{n\choose k}\;{n-k\choose j}{m\choose i,l,h}\;(y_1 +y_2)^i y^{l+j}_2 (y_1x_1-1)^h \Big\{(y_1+x_1)^{(n-k-j)}\Big\}_\mathbf{a}x^{l}_1 x^{i+k}_2 .
\end{eqnarray*}
If $T$ and $S$ are doubly commuting, then replacing $x_1$ by $T$, $x_2$ by $S$, $y_1$ by $T^*$, $y_2$ by $S^*$  and $\textbf{a}$ by $A$, we obtain (\ref{xx1}) and (\ref{xx2}).

\end{proof}

\begin{theorem}\label{bibb}
Let $A\in \mathcal{B}(H)^+$, $T \in \mathcal{B}(H)$ and $Q$ is a nilpotent operator of
order $r.$ If $T$ and $Q$ are doubly commuting, then the following statements hold.
\begin{enumerate}
\item If $T$ is $(A,(m, n))$-isosymmetric, then $(T +Q)$ is $(A,(m+2r -2, n+2r -1))$-isosymmetric.
\item If $T$ is skew $(A,(m, n))$-isosymmetric, then $(T +Q)$ is skew $(A,(m+2r -2, n+2r -1))$-isosymmetric.
\end{enumerate}
\end{theorem}

\begin{proof}
\begin{enumerate}
\item By (\ref{xx1}),
\begin{eqnarray*}
\Omega_{A}^{m+2r -2,n+2r -1}(T+Q)&=&\sum_{k=0}^{n+2r -1}\sum_{j=0}^{n+2r -1-k}\sum_{i+l+h=m+2r -2}(-1)^k{n+2r -1\choose k}\;{n+2r -1-k\choose j}\\
&&{m+2r -2\choose i,l,h}(T^*+Q^*)^i Q^{*l+j} \Omega_{A}^{h,n+2r -1-k-j}(T)T^l Q^{i+k},
\end{eqnarray*}
Note that if $(k+j)\geq2r$ or $i\geq r$ or $l\geq r$, then $(T^*+Q^*)^i Q^{*l+j} \Omega_{A}^{h,n+2r -1-k-j}(T)T^l Q^{i+k}=0$ since either $Q^{*l+j}=0$ or $Q^{i+k}=0$. On the other hand, if $(k+j)\leq2r-1$, then
$$
n+2r-1-(k+j)\geq n+2r-1-(2r-1)=n,
$$
and, if $i\leq r-1$ and $l\leq r-1$, then
$$
h=m+2r-2-i-l\geq m+2r-2-(r-1)-(r-1)\geq m.
$$
Hence, $T$ is an $(A,(h,n+2r -1-k-j))$-isosymmetry and $\Omega_{A}^{h,n+2r -1-k-j}(T)=0$. Therefore, $\Omega_{A}^{m+2r -2,n+2r -1}(T+Q)=0$ and $(T +Q)$ is $(A,(m+2r -2, n+2r -1))$-isosymmetric.
\item By (\ref{xx2}),
\begin{eqnarray*}
\Lambda_{A}^{m+2r -2,n+2r -1}(T+Q)&=&\sum_{k=0}^{n+2r -1}\sum_{j=0}^{n+2r -1-k}\sum_{i+l+h=m+2r -2}{n+2r -1\choose k}\;{n+2r -1-k\choose j}\\
&&{m+2r -2\choose i,l,h}(T^*+Q^*)^i Q^{*l+j} \Lambda_{A}^{h,n+2r -1-k-j}(T)T^l Q^{i+k}.
\end{eqnarray*}
By using this identity and arguing in the same way as in the first statement we complete the proof.
\end{enumerate}
\end{proof}
Theorem \ref{bibb} enjoys a number of interesting consequences that we now
stay. Some of these results have appeared in the literature with different proofs.
\begin{corollary}
Let $T,\,Q \in \mathcal{B}(H)$ satisfying $TQ=QT$ and $Q$ is $r$-nilpotent and fix an
integer $l\geq1$. Then, it holds:
\begin{enumerate}
\item If $T$ is $(A,m)$-isometric, then $(T + Q)^l$ is $(A,m + 2r - 2)$-isometric.
\item If $T$ is $(A,n)$-symmetric, then $(T + Q)^l$ is $(A,n + 2r - 1)$-symmetric.
\item If $T$ is skew $(A,n)$-symmetric, then $(T + Q)^l$ is skew $(A,n + 2r - 1)$-symmetric $(l$ odd$)$.
\end{enumerate}
\end{corollary}

\begin{corollary}
Let $A\in \mathcal{B}(H)^+$ and $R,\,S\in\mathcal{B}(H).$ Assume that $R$ and $S$ are doubly commuting. If $R$ is $(A,(m, n))$-isosymmetric, then the operator
 $\textbf{K}=\left(
    \begin{array}{cc}
      R & S \\
      0 & R \\
    \end{array}
  \right)
$ on $\textbf{H}=H\oplus H$ is $(A\oplus A,(m+2, n+3))$-isosymmetric.
\end{corollary}

\begin{proof}
Put
$\textbf{T}=\left(
     \begin{array}{cc}
       R & 0 \\
       0 & R \\
     \end{array}
   \right)
$ and
$\textbf{N}=\left(
     \begin{array}{cc}
       0 & S \\
       0 & 0 \\
     \end{array}
   \right)
.$ It follows that $\textbf{A}$ is positive on $\textbf{H}$, $\textbf{K}=\textbf{T}+\textbf{N}$, $\textbf{N}$ is $2$-nilpotent and $\textbf{T}$ and $\textbf{N}$ are doubly commuting. Applying Theorem \ref{bibb}, we deduce that $\textbf{K}$ is $(\textbf{A},(m+2, n+3))$-isosymmetric.

\end{proof}

\section{ A characterization of $(2\times2)$ upper triangular operator matrices in $(A;(1,1))$}
\label{S2}
For $A\in \mathcal{B}(H)^+$ and $T\in \mathcal{B}(H)$, let consider the following identities:
\begin{equation}\label{pol1}
\Omega_{A}^{1,1}(T)=T^{*}\mathcal{I}^{1}_{A}(T)-\mathcal{I}^{1}_{A}(T)T=T^{*2}AT-T^{*}AT^{2}-T^{*}A+AT,
\end{equation}
\begin{equation}\label{pol2}
\Lambda_{A}^{1,1}(T)=\mathcal{I}^{1}_{A}(T)T+T^{*}\mathcal{I}^{1}_{A}(T)T=T^{*2}AT+T^{*}AT^{2}-T^{*}A-AT.
\end{equation}
It holds, from (\ref{pol1}) and (\ref{pol2}), that
\begin{equation}\label{pol11}
\Omega_{A}^{1,1}(T)+\Lambda_{A}^{1,1}(T)=2(T^{*2}AT-T^{*}A)=2T^{*}\mathcal{I}^{1}_{A}(T),
\end{equation}
Note that $\Omega_{A}^{1,1}(T)$ is skew symmetric and $\Lambda_{A}^{1,1}(T)$ is symmetric. An operator $T$ is said to be $A$-isosymmetric (resp. skew $A$-isosymmetric) if $\Omega_{A}^{1,1}(T)=0$ (resp. $\Lambda_{A}^{1,1}(T)=0$).

\begin{example}\label{example}
\noindent
\begin{enumerate}
\item Let $H=\mathbb{C}^2$, $T = \left(
            \begin{array}{cc}
              a &  b \\
              c &  d \\
            \end{array}
          \right)
$ $(a,\,b,\,c,\,d\in \mathbb{R})$ and
$A=\left(
 \begin{array}{cc}
 0 & 0 \\
 0 & 1 \\
 \end{array}
 \right)$. It holds that,
\begin{eqnarray*}
\Omega_{I}^{1,1}(T)&=&\left(
            \begin{array}{cc}
             0 &  (b-c)(cb-ad+1 \\
              (b-c)(ad-bc-1) &  0 \\
            \end{array}
          \right),\\
\Omega_{A}^{1,1}(T)&=&\left(
            \begin{array}{cc}
             0 &  -c(bc-ad+1 \\
              -c(ad-bc-1) &  0 \\
            \end{array}
          \right).
\end{eqnarray*}
\begin{enumerate}
\item $T$ is $I$-isosymmetric if and only if \big($b=c$ or $ad-bc=1$\big).
\item $T$ is $A$-isosymmetric if and only if \big($c=0$ or $ad-bc=1$\big).
\end{enumerate}
\item Let $H=\mathbb{C}^2$, $T = \left(
            \begin{array}{cc}
              a &  b \\
              0 &  d \\
            \end{array}
          \right)
$ $(a,\,b,\,d\in \mathbb{R})$ and
$A=\left(
 \begin{array}{cc}
 0 & 0 \\
 0 & 1 \\
 \end{array}
 \right)$. It holds from (\ref{pol11}) that,
\begin{eqnarray*}
\Lambda_{I}^{1,1}(T)&=&\left(
            \begin{array}{cc}
             2a(a^2-1) &  b(2a^2-1+ad) \\
            b\big( 2a^2+ad-1\big) &  2(b^2(a+d)+d^3-d) \\
            \end{array}
          \right),\\
\Lambda_{A}^{1,1}(T)&=&\left(
            \begin{array}{cc}
              0 & 0 \\
              0 & 2d(1-d^2)\\
            \end{array}
          \right).
\end{eqnarray*}
\begin{enumerate}
\item $T$ is skew $I$-isosymmetric if and only if one of the following situations holds true
\begin{enumerate}
\item $a=0$, $b=0$ and $d=\pm1$.
\item $a=1$, $b=0$ and \big($d=0$ or $d=\pm1$\big).
\item $a=1$ and $d=-1.$
\item $a=-1$, $b=0$ and \big($d=0$ or $d=\pm1$\big).
\item $a=-1$ and $d=1.$
\end{enumerate}
\item $T$ is skew $A$-isosymmetric if and only if \big($d=0$ or $d=\pm1$\big).
\end{enumerate}
\end{enumerate}
\end{example}

\begin{remark} Let $A\in \mathcal{B}(H)^+$ and $T\in \mathcal{B}(H)$.
\begin{enumerate}
\item $T$ is an $A$-isosymmetry (resp. a skew $A$-isosymmetry) if and only if $-T$ is an $A$-isosymmetry (resp. a skew $A$-isosymmetry).
\item If $T$ is $A$-expansive, then: $T$ is $A$-isosymmetric (resp. skew $A$-isosymmetric) if and only if $T$ is $\mathcal{I}^{1}_{A}(T)$-symmetric (resp. skew $\mathcal{I}^{1}_{A}(T)$-symmetric).
\item The following statements are equivalent
\begin{enumerate}
\item $T$ is $A$-isosymmetric.
\item $T^{*}\mathcal{S}^{1}_{A}(T)T=\mathcal{S}^{1}_{A}(T).$
\item $T^{*}\mathcal{I}^{1}_{A}(T)=\mathcal{I}^{1}_{A}(T)T.$
\end{enumerate}
\item For an invertible operator $A$, let
$$R:=\frac{1}{2}\big(T+A^{-1}T^* A\big), \qquad S:=\frac{1}{2}\big(T-A^{-1}T^* A\big).$$
Then, $R$ is $A$-isosymmetric, $S$ is skew $A$-isosymmetric and $T=R+S$.
\end{enumerate}
\end{remark}

In the following results, we consider a class of $(2\times2)$ upper triangular operator matrices. We aim to extend \cite[Lemma 4.10, Lemma 4.27]{mark1}, \cite[Theorem 2.10, Corollary 2.11, Theorem 2.12]{t1} established for isosymmetries and complex isosymmetries, respectively, to our developed approach about $A$-isosymmetries.
\begin{theorem}Let $A\in \mathcal{B}(H)^+$, $N,E,X\in \mathcal{B}(H)$ and $\textbf{T}=\left(
                                                                                            \begin{array}{cc}
                                                                                              N & E \\
                                                                                              0 & X \\
                                                                                            \end{array}
                                                                                          \right)
$ on $\textbf{H}=H\oplus H$. Then, the following statements hold.
\begin{enumerate}
\item Assume that $E=NEX$, $N^*AE=AEX$ and $N$ is $A$-isometric. Then, $\textbf{T}$ is $A$-isosymmetric if and only if $X$ is $A$-isosymmetric.
\item Assume that $N$ is $A$-symmetric and $NE=EX$. Then, $\textbf{T}$ is $A$-isosymmetric if and only if $X$ is $A$-isosymmetric and $AE=ANEX$ holds.
\end{enumerate}
\end{theorem}

\begin{proof}
\begin{enumerate}
\item Since $A\in \mathcal{B}(H)^+$, we have $\textbf{A}=A\oplus A \in \mathcal{B}\big(H\oplus H\big)^+.$ On the other hand,
$$
\mathcal{I}^{1}_{A}(\textbf{T})
=\left(
   \begin{array}{cc}
     \mathcal{I}^{1}_{A}(N) & N^* AE \\
     E^*AN & E^*AE+\mathcal{I}^{1}_{A}(X) \\
   \end{array}
 \right).
$$
It holds from that
\begin{eqnarray}\label{riif1}
\textbf{T}^*\mathcal{I}^{1}_{A}(\textbf{T})
&=&\left(
  \begin{array}{cc}
    N^*\mathcal{I}^{1}_{A}(N)  & N^{*2}AE \\
    E^*\mathcal{I}^{1}_{A}(N)+X^*E^*AN  & E^*N^*AE+X^*E^*AE+X^*\mathcal{I}^{1}_{A}(X) \\
  \end{array}
\right),
\end{eqnarray}

\begin{eqnarray}\label{riif2}
\mathcal{I}^{1}_{A}(\textbf{T})\textbf{T}
&=&\left(
   \begin{array}{cc}
     \mathcal{I}^{1}_{A}(N)N & \mathcal{I}^{1}_{A}(N)E+N^*AEX \\
     E^*AN^2 & E^*ANE+E^*AEX+\mathcal{I}^{1}_{A}(X)X \\
   \end{array}
 \right).
\end{eqnarray}
By equations (\ref{riif1}) and (\ref{riif2}), $\textbf{T}$ is $A$-isosymmetric if and only if
\begin{equation}\label{riif3}
\left\{
  \begin{array}{ll}
    N^*\mathcal{I}^{1}_{A}(N)=\mathcal{I}^{1}_{A}(N)N, & \hbox{ } \\
    E^*\mathcal{I}^{1}_{A}(N)+X^*E^*AN=E^*AN^2, & \hbox{ } \\
    N^{*2}AE=\mathcal{I}^{1}_{A}(N)E+N^*AEX, & \hbox{} \\
    \big(E^*N^*AE+X^*E^*AE\big)+X^*\mathcal{I}^{1}_{A}(X)=\big(E^*ANE+E^*AEX\big)+\mathcal{I}^{1}_{A}(X)X. & \hbox{}
  \end{array}
\right.
\end{equation}
Assume that $N$ is $A$-isometric, that is $\mathcal{I}^{1}_{A}(N)=N^*AN-A=0$. It holds from that
\begin{equation}\label{riif4}
\left\{
  \begin{array}{ll}
    X^*E^*AN=E^*AN^2, & \hbox{ } \\
    N^{*2}AE=N^*AEX, & \hbox{} \\
    \big(E^*N^*AE+X^*E^*AE\big)+X^*\mathcal{I}^{1}_{A}(X)=\big(E^*ANE+E^*AEX\big)+\mathcal{I}^{1}_{A}(X)X. & \hbox{}
  \end{array}
\right.
\end{equation}
If $E=NEX$ and $N^*AE=AEX$, then the first and the second equation of (\ref{riif4}) hold. Moreover, we have
\begin{eqnarray*}
E^*ANE+E^*AEX&=&X^*E^*\underbrace{N^*AN}E+E^*N^*AE\\
&=&X^*E^*AE+E^*N^*AE.
\end{eqnarray*}
The above equation gives
$$X^*\mathcal{I}^{1}_{A}(X)=\mathcal{I}^{1}_{A}(X)X.$$
Hence, $X$ is $A$-isosymmetric. Arguing as above, we can show the converse of the claim.
\item Assume that $N$ is $A$-symmetric, that is $\mathcal{S}^{1}_{A}(N)=N^*A-AN=0$ and $\textbf{T}$ is $A$-isosymmetric. From the assumptions, it holds $X^*E^*A=E^*N^*A=E^*AN,$ and so $N^*AE=AEX.$ Hence, it follows
\begin{equation}\label{riif5}
\left\{
  \begin{array}{ll}
   X^*E^*AE=E^*ANE, & \hbox{} \\
   E^*N^*AE=E^*AEX. & \hbox{}
  \end{array}
\right.
\end{equation}
\end{enumerate}
From (\ref{riif3}) and (\ref{riif5}), we obtain
$$E^*N^*AE+X^*E^*AE=E^*ANE+E^*AEX. $$
So, $X^*\mathcal{I}^{1}_{A}(X)=\mathcal{I}^{1}_{A}(X)X.$ Hence, $X$ is $A$-isosymmetric. On the other hand, (\ref{riif3}) becomes
$$
\left\{
  \begin{array}{ll}
    E^*(N^{*2}-I)A+X^*E^*N^* A=E^*N^{*2}A, & \hbox{} \\
    AN^{2}E=A(N^2-I)E+ANEX. & \hbox{}
  \end{array}
\right.
$$
This gives that
$$
\left\{
  \begin{array}{ll}
    E^*A=X^*E^*N^* A, & \hbox{} \\
    AE=ANEX. & \hbox{}
  \end{array}
\right.
$$
By similar arguments, we can show the other implication.

\end{proof}

\begin{corollary}
Let $A\in \mathcal{B}(H)^+$, $N,E,X\in \mathcal{B}(H)$ and $\textbf{T}=\left(
                                                                                            \begin{array}{cc}
                                                                                              N & E \\
                                                                                              0 & X \\
                                                                                            \end{array}
                                                                                          \right)
$ on $\textbf{H}=H\oplus H$ with $N$ an $A$-isometry. If $N^*AE=0$, then $\textbf{T}$ is $A$-isosymmetric.
\end{corollary}

\begin{proof}
Since $N$ is an $A$-isometry and $N^*AE=0$, it holds from (\ref{riif4}) that
$$X^*(E^*AE+X^*AX-A)=(E^*AE+X^*AX-A)X.$$
On the other hand,
$$\hbox{\big($T$ is $A$-isosymmetric\big)}\Longleftrightarrow \big(X^*(E^*AE+X^*AX-A)=(E^*AE+X^*AX-A)X\big).$$
This completes the proof.

\end{proof}

In the following theorem, we describe the block operator form of an $A$-isosymmetry such that $\mathcal{M}=\ker\big(\mathcal{I}^{1}_{A}(\textbf{T})\big)\neq\{0\}$.

\begin{theorem}
Let $\textbf{T} \in \mathcal{B}(H)$ and $A\in \mathcal{B}(H)^+$. Assume that $\mathcal{M}$ is invariant for $A$ and for $\textbf{T}$. Then $\textbf{T}$ has the form
\begin{equation}\label{riif6}
\textbf{T}=\left(
                                                                                            \begin{array}{cc}
                                                                                              N & E \\
                                                                                              0 & X \\
                                                                                            \end{array}
                                                                                          \right)\;\;\hbox{on $\mathcal{M}\oplus \mathcal{M}^\perp$},\quad N\in \mathcal{B}(\mathcal{M}),\,E\in\mathcal{B}(\mathcal{M}^\perp,\,\mathcal{M}),X\in \mathcal{B}(\mathcal{M}^\perp),
\end{equation}
such that $N$ is an $A_{|\mathcal{M}}$-isometry and $E^*A_{|\mathcal{M}}N=0.$
\end{theorem}

\begin{proof}Since $\mathcal{M}$ is invariant for $A$, it follows that $\mathcal{M}^\perp$ is invariant for $A.$ Moreover, since $A\in \mathcal{B}(H)^+$, we have  $A_1=A_{|\mathcal{M}}\geq0$, $A_2=A_{|\mathcal{M}^\perp}\geq0$ and $A=A_1\oplus A_2.$ If $\mathcal{M}$ is invariant for $\textbf{T}$, then $\textbf{T}$ has the block operator form (\ref{riif6}). Hence it follow that
$$
\mathcal{I}^{1}_{A}(\textbf{T})
=\left(
   \begin{array}{cc}
     N^*A_1N-A_1 & N^* A_1E \\
     E^*A_1N & E^*A_1E+X^*A_2X-A_2 \\
   \end{array}
 \right)\quad\hbox{on $\mathcal{M}\oplus \mathcal{M}^\perp$}.
$$
Since $\mathcal{M}=\ker\big(\mathcal{I}^{1}_{A}(T)\big)$, it holds $\mathcal{I}^{1}_{A}(T)(x\oplus0)=0$ for $x\in \mathcal{M}.$ Moreover, we have $N^*A_1N-A_1=0$ and $E^*A_1N =0$ on $\mathcal{M}$. Hence, $N$ is an $A_1$-isometry.

\end{proof}

\end{document}